\documentclass[english,british,12pt]{paper}
\usepackage[T1]{fontenc}
\usepackage[latin9]{inputenc}
\usepackage{geometry}
\usepackage{units}
\usepackage{enumitem}
\usepackage{amsmath}
\usepackage{amsthm}
\usepackage{amssymb}
\usepackage{graphicx}
\usepackage{esint}
\usepackage[numbers]{natbib}

\makeatletter

\providecommand{\tabularnewline}{\\}

\newcommand{\lyxaddress}[1]{
\par {\raggedright #1
\vspace{1.4em}
\noindent\par}
}
  \theoremstyle{definition}
  \newtheorem{defn}{\protect\definitionname}
\theoremstyle{plain}
\newtheorem{thm}{\protect\theoremname}
  \theoremstyle{remark}
  \newtheorem{rem}{\protect\remarkname}

\usepackage{color}
\usepackage{hyperref}

\makeatother

\usepackage{babel}
  \addto\captionsbritish{\renewcommand{\definitionname}{Definition}}
  \addto\captionsbritish{\renewcommand{\remarkname}{Remark}}
  \addto\captionsenglish{\renewcommand{\definitionname}{Definition}}
  \addto\captionsenglish{\renewcommand{\remarkname}{Remark}}
  \providecommand{\definitionname}{Definition}
  \providecommand{\remarkname}{Remark}
\addto\captionsbritish{\renewcommand{\theoremname}{Theorem}}
\addto\captionsenglish{\renewcommand{\theoremname}{Theorem}}
\providecommand{\theoremname}{Theorem}

\begin{document}

\title{Nonlinear model identification and spectral submanifolds for multi-degree-of-freedom
mechanical vibrations}

\author{Robert Szalai$^{1}$\thanks{Corresponding author. Email: r.szalai@bristol.ac.uk},
David Ehrhardt$^{2}$ and George Haller$^{3}$}
\maketitle

\lyxaddress{$^{1}$Department of Engineering Mathematics, University of Bristol,
Merchant Venturers Building, Woodland Road, Bristol BS8 1UB, United
Kingdom}

\lyxaddress{$^{2}$Department of Mechanical Engineering, University of Bristol,
Queen's Building, University Walk, Clifton BS8 1TR, United Kingdom}

\lyxaddress{$^{3}$Institute for Mechanical Systems, ETH Zürich, Leonhardstrasse
21, Zürich, 8092, Switzerland.}
\begin{abstract}
In a nonlinear oscillatory system, spectral submanifolds (SSMs) are
the smoothest invariant manifolds tangent to linear modal subspaces
of an equilibrium. Amplitude-frequency plots of the dynamics on SSMs
provide the classic backbone curves sought in experimental nonlinear
model identification. We develop here a methodology to compute analytically
both the shape of SSMs and their corresponding backbone curves from
a data-assimilating model fitted to experimental vibration signals.
Using examples of both synthetic and real experimental data, we demonstrate
that this approach reproduces backbone curves with high accuracy. 
\end{abstract}

\section{Introduction}

Modal decomposition into normal modes is a powerful tool in linear
system identification (see., e.g., Ewins \citep{EwinsBook}), but
remains inapplicable to nonlinear systems due to the lack of a superposition
principle. Various nonlinear normal mode (NNM) concepts nevertheless
offer a conceptual simplification in the description of small-amplitude
nonlinear vibrations.

For conservative oscillatory systems with no resonance, the Lyapunov
subcenter-manifold theorem (Kelley \citep{KELLEY1967472}) guarantees
the existence of a unique, analytic surface of periodic orbits that
is tangent to any selected two-dimensional modal subspace (or eigenspace)
of the linearised system at the origin. Each periodic orbit in such
a subcenter manifold is a NNM by the classic definition of Rosenberg
\citep{Rosenberg66}. In contrast, Shaw and Pierre \citep{ShawPierre}
call the subcenter manifold itself a NNM. 

Shaw and Pierre \citep{ShawPierre} also extend the latter view to
dissipative systems, envisioning NNMs as invariant manifolds tangent
to modal subspaces of an equilibrium point (see the reviews of Vakakis
\citep{vakakis2001normal}, Kerschen et al. \citep{kerschen2009nonlinear},
Peeters et al. \citep{Peeters}, and Avramov and Mikhlin \citep{mikhlin2010nonlinears,avramov2013review}).
As observed recently, however, by multiple authors (Neild et al. \citep{Neild20140404},
Cirillo et al. \citep{Cirillo2015}, and Haller and Ponsioen \citep{Haller2016}),
such invariant manifolds are non-unique even for linear systems, let
alone for nonlinear ones. Formal Taylor expansions and operational
numerical procedures do nevertheless yield approximate invariant surfaces
in most problems. This effectiveness of the Shaw\textendash Pierre
approach has inspired its formal extension to invariant manifolds
modelled over multiple modes (Pescheck et al. \citep{Pesheck2001}),
as well as to time-dependent invariant manifolds under external harmonic
forcing (see, e.g., Jiang, Pierre and Shaw \citep{jiang2005nonlinear},
Gabale and Sinha \citep{Gabale20112596}).

In a recent mathematical treatment, Haller and Ponsioen \citep{Haller2016}
unites the Rosenberg and Shaw\textendash Pierre NNM concepts for dissipative
systems under possible time-dependent forcing. In this setting, a
\emph{nonlinear normal mode} (NNM) is a near-equilibrium oscillation
with finitely many frequencies. This NNM concept includes the trivial
case of an equilibrium with no (i.e. zero) oscillation frequencies;
Rosenberg's case of a periodic orbit; and the case of a quasiperiodic
oscillation with finitely many rationally independent frequencies.
Haller and Ponsioen \citep{Haller2016} then defines a \emph{spectral
submanifold} (SSM) as the smoothest invariant manifold tangent to
a spectral subbundle along a NNM. For a trivial NNM (equilibrium),
a spectral subbundle is a modal subspace of the linearised system
at the equilibrium, and hence an SSM is the smoothest Shaw\textendash Pierre-type
invariant manifold tangent to this modal subspace. Similarly, for
periodic or quasiperiodic NNMs, an SSM is the smoothest invariant
manifold among those sought formally in time-dependent extensions
of the Shaw-Pierre surfaces (see Fig. (\ref{fig:Node}) for illustration).

\begin{figure}
\begin{centering}
\includegraphics[width=0.7\textwidth]{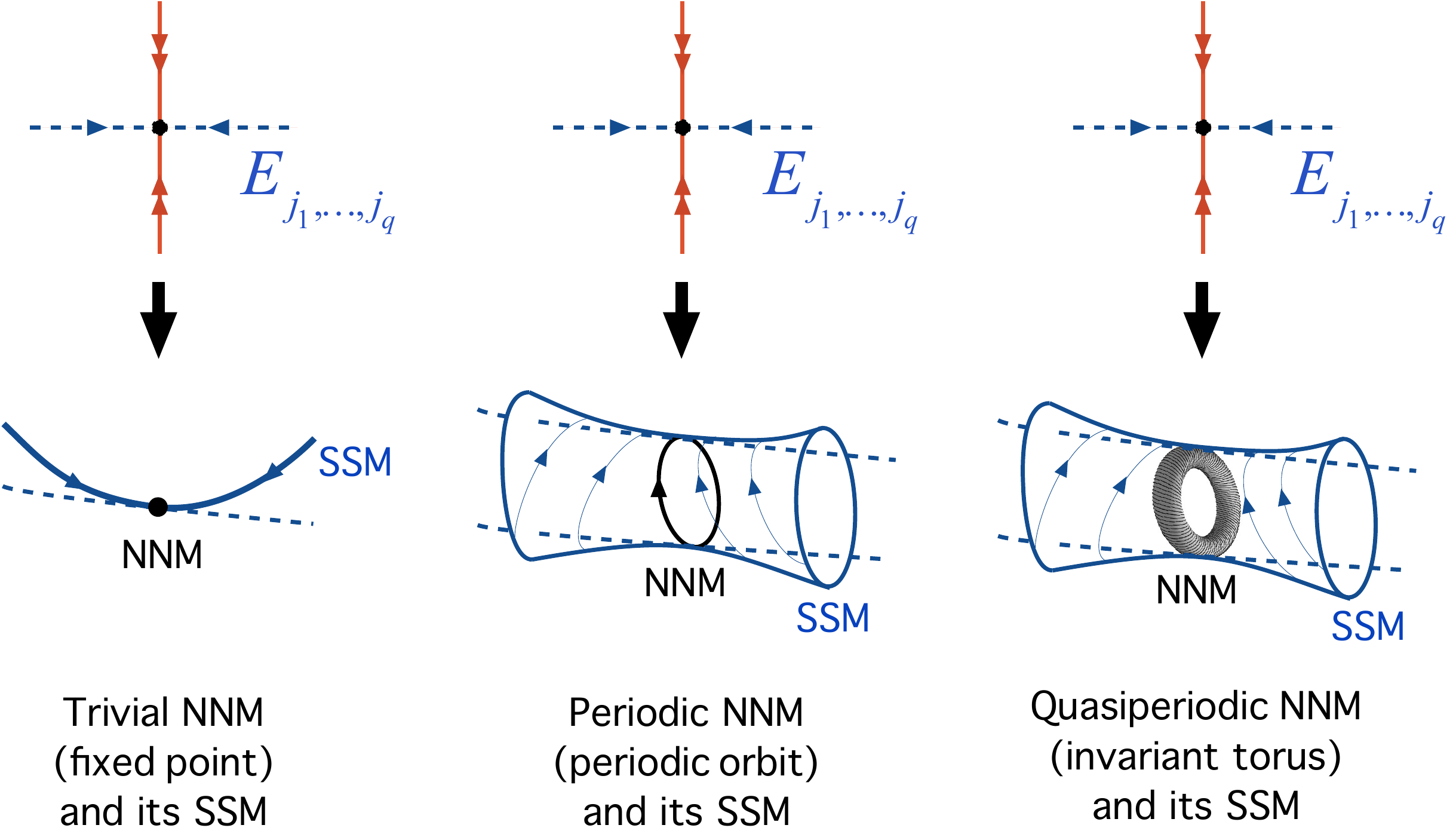}
\par\end{centering}
\caption{The three main types of NNMs (trivial, periodic and quasiperiodic)
and their corresponding SSMs (autonomous, periodic and quasiperiodic).
The NNMs are, or are born out of perturbations of, a fixed point.
The SSMs are, in contrast, the smoothest invariant manifolds tangent
to a subbundle along the NNM whose fibers are close to a specific
modal subspace $E_{j_{1},\ldots,j_{q}}$ of the linearised system.
Here the indices $j_{1},\ldots,j_{q}$ refer to an arbitrary selection
of $q$ two-dimensional modal subspaces of the linearised system (cf.
\citep{Haller2016} for more detail). \label{fig:Node}}
\end{figure}

Here we adopt the above distinction between NNMs and SSMs and restrict
our attention to SSMs of trivial NNMs (fixed points). Even in this
simplest setting, it is not immediate that a single smoothest invariant
manifold tangent to a modal subspace of the fixed point actually exists.
This question, however, is positively answered under certain conditions
by the abstract invariant manifold results of Cabré et al. \citep{CabreLlave2003},
as explained by Haller and Ponsioen \citep{Haller2016}. These results
also provide a computationally efficient way of computing SSMs using
the parametrisation method (cf. Haro et al. \citep{Haro2016} for
a general introduction). 

The reduced dynamics on a single-mode SSM gives an exact nonlinear
extension of the linear dynamics of the modal subspace to which the
SSM is tangent. This extension is characterised by a \emph{backbone
curve,} i.e., a graph expressing the instantaneous vibration amplitude
as a function of the instantaneous vibration frequency along the SSM. 

Without specific concern for SSMs, backbone curves have been approximated
operationally in a number of numerical studies\textbf{.} One approach
assumes that the mechanical system is conservative apart from a weak
damping term that is a linear (or at least odd), position-independent
function of the velocities. In such a system, a periodic forcing producing
a $90^{\circ}$ out-of-phase response preserves exactly a periodic
orbit (i..e, Rosenberg's NNM) of the conservative limit (Peeters et
al. \citep{Peeters2011a,Peeters2011b}). The systematic construction
of external forcing that yields the required $90^{\circ}$ phase lag
for various frequencies is usually referred to as the\emph{ force
appropriation method. }In practice, force appropriation involves a
tedious tuning process that also suffers from unintended interactions
between a shaker and the nonlinear system. 

To expedite the backbone curve construction, one may locate a single
high-amplitude periodic NNM from force appropriation, then turn off
the forcing, and identify \textendash{} by signal processing \textendash{}
the instantaneous amplitude-frequency relation of the decaying vibration
as the backbone curve. Usually referred to as the \emph{resonance
decay method}, this process tacitly assumes that the decaying vibrations
closely follow a Lyapunov subcenter manifold of the conservative limit.
In our terminology, the assumption is that the analytic subcenter
manifold of the conservative limit perturbs smoothly to a unique SSM
under small enough damping. While this statement seems exceedingly
difficult to establish mathematically, it appears to hold true for
small enough viscous damping (see Kerschen et a. \citep{Kerschen2006505}
and Peeters et al. \citep{Peeters2011a}). Therefore, under weak viscous
damping, the resonance decay approach gives consistent results for
SSMs, provided that the decaying vibrations are close to the (yet
unknown) SSM. Small errors in this initialisation are expected to
be persistent for fast SSMs, i.e., SSMs tangent to the modal subspaces
with higher damping. This is because the off-SSM components (errors)
in the initial conditions decay much slower than the in-SSM components
(useful signal), which adds substantial inaccuracy to the backbone
curve construction.

A third approach to backbone-curve construction uses time-dependent
normal forms to construct approximate reduced-order nonlinear models
of the system near each natural frequency. The backbone curve is then
obtained approximately by the method of harmonic balance applied to
the reduced model under resonant parametric forcing (see, e.g.,\textbf{
}\citep{Neild20140404}). An advantage of this method is its ability
to deal with internal resonances, producing intricate multi\textendash dimensional
backbone surfaces. The underlying assumption for all this is that
the higher-order normal form terms coupling the reduced model to the
remaining degrees of freedom are small, and that the oscillations
have small enough amplitudes for the harmonic balance method to be
reasonably accurate. A further important assumption is that an exact
nonlinear model for the system is available for the purposes of computing
a normal form. This tends to limit the practical use of this approach
to simple geometries and materials.

In summary, several methods for numerical or experimental backbone-curve
construction are available, but all make assumptions limiting their
range of applicability. These assumptions include small, position-independent,
and linear viscous damping; small enough oscillation amplitudes; an
accurate initial knowledge of the SSM; and yet unproven results on
the smooth persistence of Lyapunov subcenter manifolds as SSMs under
nonzero damping.

Here we develop a backbone-curve identification method that addresses
most of the above challenges. We infer backbone curves directly from
the dynamics on the SSM, without making any assumptions on the damping
or any reference to the Lyapunov subcenter manifold of the conservative
limit. As input, we assume that tracks of decaying vibration data
are available in the vicinity of $N$ natural frequencies. We simultaneously
assimilate all this data into a nonlinear discrete mapping model of
the near-equilibrium dynamics of the system. We then construct backbone
curves analytically from the nonlinear dynamics on the SSMs of this
discrete mapping. 

We illustrate the generality and accuracy of this approach on two
examples. Our first example is a two-degree-of-freedom nonlinear mechanical
system, for which we perform both an analytic and a data-assimilating
construction of the backbone curves for comparison. Our second example
is a clamped-clamped beam experiment \citep{Ehrhardt2016612}, in
which we determine the first three SSMs simultaneously from measurements
of decaying vibration signals.

\section{Set-up\label{sec:Set-up}}

We start with an $n$-degree of freedom, autonomous mechanical system
of the general form
\begin{equation}
\boldsymbol{M}(\boldsymbol{q})\ddot{\boldsymbol{q}}-\boldsymbol{f}(\boldsymbol{q},\dot{\boldsymbol{q}})=\boldsymbol{0},\qquad\boldsymbol{f}(\boldsymbol{0},\boldsymbol{0})=\boldsymbol{0},\label{eq:starting point}
\end{equation}
where the mass matrix $\boldsymbol{M}(\boldsymbol{q})\in\mathbb{R}^{n\times n}$
and its inverse $\boldsymbol{M}^{-1}(\boldsymbol{q})$ are of class
$C^{r}$, with $r\geq1$, in the generalised coordinate vector $\boldsymbol{q}\in\mathbb{R}^{n}$.
The forcing vector $\boldsymbol{f}(\boldsymbol{q})\in\mathbb{R}^{n}$
is also $C^{r}$ in its arguments, containing all conservative and
non-conservative autonomous forces, both linear and nonlinear. 

Beyond taking nonnegative integer values, the smoothness parameter
$r$ is also allowed to be $r=\infty$ (arbitrarily many times differentiable
functions) or $r=a$ (analytic functions, i.e., $C^{\infty}$ functions
with a convergent Taylor expansion in a complex neighbourhood of $(\boldsymbol{q},\dot{\boldsymbol{q}})=(\boldsymbol{0},\boldsymbol{0})$).
The degree of freedom $n\geq1$ is allowed to be arbitrarily high
and may also be in principle infinity (continuum vibrations), although
some of our assertions about properties of the solutions would need
to be verified on a case-by-case basis in the infinite-dimensional
setting. By the formulation in (\ref{eq:starting point}), $\boldsymbol{q}\equiv\boldsymbol{0}$
is an equilibrium point for the system.

The equivalent first-order form of the differential equation (\ref{eq:starting point})
is obtained by letting $\boldsymbol{x}=(\boldsymbol{q},\dot{\boldsymbol{q}})\in\mathbb{R}^{2n},$
which leads to 
\begin{equation}
\dot{\boldsymbol{x}}=\boldsymbol{\mathcal{F}}(\boldsymbol{x}),\qquad\boldsymbol{\mathcal{F}}(\boldsymbol{0})=\boldsymbol{0,}\qquad\boldsymbol{\mathcal{F}}(\boldsymbol{x})=\left(\begin{array}{c}
\dot{\boldsymbol{q}}\\
\boldsymbol{M}^{-1}(\boldsymbol{q})\boldsymbol{f}(\boldsymbol{q},\dot{\boldsymbol{q}})
\end{array}\right),\label{eq:firstorder}
\end{equation}
where $\boldsymbol{\mathcal{F}}(\boldsymbol{x})\in\mathbb{R}^{2n}$
is $C^{r}$ in its arguments. The solutions $\boldsymbol{x}(t)$ of
(\ref{eq:firstorder}) give rise to the flow map 
\[
\boldsymbol{\varXi}_{t}\colon\boldsymbol{x}_{0}\mapsto\boldsymbol{x}(t),
\]
where $\boldsymbol{x}_{0}=\boldsymbol{x}(0)$. 

The linearisation of (\ref{eq:firstorder}) at the equilibrium point
$\boldsymbol{x}=\boldsymbol{0}$ is given by
\begin{equation}
\dot{\boldsymbol{x}}=\boldsymbol{\mathcal{{A}}}\boldsymbol{x},\qquad\boldsymbol{\mathcal{{A}}}=D\boldsymbol{\mathcal{F}}(\boldsymbol{0}).\label{eq:linearized}
\end{equation}
We assume that $\boldsymbol{\mathcal{{A}}}$ has $n$ pairs of complex
conjugate eigenvalues $\lambda_{1},\bar{\lambda}_{1},\ldots,\lambda_{n},\bar{\lambda}_{n}$,
satisfying 
\begin{equation}
\mathrm{Re}\lambda_{n}\leq\ldots\leq\mathrm{Re}\lambda_{1}<0,\label{eq:assumption on eigenvalues of ODE}
\end{equation}
and hence the equilibrium point is linearly asymptotically stable.
This context is relevant for underdamped structural vibrations, in
which the nonlinear system (\ref{eq:starting point}) is known to
have a stable equilibrium, but the exact nature of its nonlinearities
is unknown.

\section{Sampled nonlinear vibrations\label{sec:Sampled-nonlinear-vibrations}}

To reduce the complexity of the flow map $\boldsymbol{\varXi}_{t}$
in our analysis, we will focus on temporally sampled approximations
to $\boldsymbol{\varXi}_{t}$. Iterating such discrete approximations,
one can still reproduce the main features of the nonlinear dynamics
at regular time intervals. Constructing the sampled dynamics via a
stroboscopic (or Poincaré) map is, however, only feasible when the
full dynamical system (\ref{eq:starting point}) is precisely known,
and hence trajectories from arbitrary initial conditions can be generated.
In practice, this is generally not the case.

Instead, we seek to reconstruct a sampled representation of $\boldsymbol{\varXi}_{t}$
from a limited number of observations of trajectories. The scalar
observable along trajectories can be, for instance, a position or
a velocity coordinate of a certain material point of the mechanical
system (\ref{eq:starting point}). We denote this observable by $\varphi(\boldsymbol{x}):\mathbb{R}^{2n}\to\mathbb{R}$,
i.e., as a scalar function of the state variable $\boldsymbol{x}$
alone. We then build a new state vector $\boldsymbol{\xi}\in\mathbb{R}^{2\nu}$
out of $2\nu$ subsequent observations along trajectories of (\ref{eq:starting point})
by letting 
\begin{equation}
\boldsymbol{\xi=}\boldsymbol{\Phi}(\boldsymbol{x}),\qquad\boldsymbol{\Phi}(\boldsymbol{x}):=\left(\varphi(\boldsymbol{x}),\varphi\left(\boldsymbol{\varXi}_{T}(\boldsymbol{x})\right),\ldots,\varphi\left(\boldsymbol{\varXi}_{T}^{2\nu-1}(\boldsymbol{x})\right)\right)\in\mathbb{R}^{2\nu},\qquad\nu\geq1.\label{eq:defxi}
\end{equation}
We have selected the dimension of $\boldsymbol{\xi}$ to be even (i.e.,
$2\nu$) to ensure basic spectral compatibility between the dynamics
of $\boldsymbol{\xi}$ and the dynamics of $\boldsymbol{x}$, as discussed
in more detail below. 

A sampling map $\boldsymbol{F}$ can be defined as the discrete mapping
advancing the current $2\nu$ observations by one, i.e., from the
observation vector $\boldsymbol{\Phi}(\boldsymbol{x})$ to the observation
vector $\boldsymbol{\Phi}(\boldsymbol{\varXi}_{T}(\boldsymbol{x}))$.
Specifically, we define the mapping $\boldsymbol{F}:\mathbb{R}^{2\nu}\to\mathbb{R}^{2\nu}$
via the relation 
\begin{equation}
\boldsymbol{\xi}_{k+1}:=\boldsymbol{F}(\boldsymbol{\xi}_{k})=\boldsymbol{F}^{k}(\boldsymbol{\xi}_{0}),\qquad k\in\mathbb{N},\qquad\boldsymbol{\xi}_{0}=\boldsymbol{\Phi}(\boldsymbol{x}_{0}),\label{eq:DISCeom}
\end{equation}
or, equivalently, as 
\begin{equation}
\boldsymbol{\Phi}\circ\boldsymbol{\varXi}_{T}=\boldsymbol{F}\circ\boldsymbol{\Phi}.\label{eq:conjugacy}
\end{equation}

By construction, the $\boldsymbol{x}=\boldsymbol{0}$ equilibrium
point of system (\ref{eq:starting point}) is mapped into a fixed
point $\boldsymbol{\xi^{0}}=\Phi(\boldsymbol{0})$ of the sampling
map $\boldsymbol{F}$ under $\boldsymbol{\Phi}$. If necessary, we
shift the $\boldsymbol{\xi}$ coordinates as $\boldsymbol{\xi}\to\boldsymbol{\xi}-\boldsymbol{\xi^{0}}$
to achieve $\boldsymbol{\xi^{0}}=\boldsymbol{0}$. Therefore, without
loss of generality, we may assume 
\begin{equation}
\boldsymbol{F}(\boldsymbol{0})=\boldsymbol{0}.\label{eq:F(0)=00003D0}
\end{equation}
 As a consequence, whenever $\varphi\in C^{r}$ holds, a Taylor expansion
of $\boldsymbol{F}$ at the origin must be of the form
\begin{equation}
\boldsymbol{F}(\boldsymbol{x})=\sum_{\left|\boldsymbol{m}\right|=1}^{r}\boldsymbol{a}_{\boldsymbol{m}}\xi_{1}^{m_{1}}\cdot\ldots\cdot\xi_{2\nu}^{m_{2\nu}}+o(\left|\boldsymbol{x}\right|^{r})=\sum_{\left|\boldsymbol{m}\right|=1}^{r}\boldsymbol{a}_{\boldsymbol{m}}\boldsymbol{\xi}^{\boldsymbol{m}}+o(\left|\boldsymbol{x}\right|^{r})\label{eq:expansion of F}
\end{equation}
for appropriate coefficient vectors $\boldsymbol{a}_{\boldsymbol{m}}\in\mathbb{R}^{2\nu}$
and integer index vector $\boldsymbol{m}=(m_{1},\ldots,m_{2\nu})\in\mathbb{N}^{2\nu}$,
whose norm we measure as $\left|\boldsymbol{m}\right|=\sum_{i=1}^{2\nu}m_{i}.$
We have used here the short-hand notation $\boldsymbol{\xi}^{\boldsymbol{m}}=\xi_{1}^{m_{1}}\cdot\ldots\cdot\xi_{2\nu}^{m_{2\nu}}$. 

\section{Delay embedding\label{sec:Delay-embedding}}

The definition (\ref{eq:DISCeom}) does not immediately clarify the
relation between the dynamics of the flow map $\boldsymbol{\varXi}_{T}$
and the dynamics of the sampling map $\boldsymbol{F}$. The Takens
Embedding Theorem \citep{Takens1981}, however, guarantees that such
a relationship exists on invariant manifolds of generic flow maps
$\boldsymbol{\varXi}_{T}$, at least for generic observables $\varphi$,
as long as the sample length $2\nu$ is long enough. 

Specifically, if $W$ is a compact, $d$-dimensional inflowing-invariant
manifold \citep{Fenichel} of system (\ref{eq:starting point}) and
$2\nu\geq2d+1$ holds, then the set of function pairs $(\boldsymbol{\varXi}_{T},\varphi)$
for which $\boldsymbol{\Phi}(W)$ is diffeomorphic to $W$ is open
and dense in the product space $\mathcal{D}^{r}(W)\times C^{r}(W,\mathbb{R}).$
Here $\mathcal{D}^{r}(W)$ denotes the space of $C^{r}$ diffeomorphisms
of $W$, and $C^{r}(W,\mathbb{R})$ denotes the space of $C^{r}$
scalar functions defined on $W$, with both spaces endowed with the
$C^{r}$ topology. 

Takens's theorem can further be strengthened (Huke \citep{huke2006embedding},
Stark \citep{Stark1999})\textbf{ }when $\boldsymbol{\varXi}_{T}$
has only a finite number of periodic orbits of periods less than $2\nu$,
with all periodic orbit admitting distinct Floquet multipliers. In
this case, for \emph{any} $\boldsymbol{\varXi}_{T},$ there is an
open and dense set of observables $\varphi\in C^{r}(W,\mathbb{R})$
such that $\boldsymbol{\Phi}$ is an embedding of $W$ into $\mathbb{R}^{2\nu}$.
This version of the theorem is particularly helpful in our setting,
as close enough to its asymptotically stable equilibrium at $\boldsymbol{x}=\boldsymbol{0}$,
the flow map $\boldsymbol{\varXi}_{T}$ will have no periodic orbits.
Therefore, it is enough for us to require the observable $\varphi$
to be generic, without having to assume anything further for $\boldsymbol{\varXi}_{T}$.
This simplification holds true on any extended neighbourhood of the
origin that has the required low number of nondegenerate periodic
orbits discussed above.

For such generic observables, $\boldsymbol{\Phi}(W)\subset\mathbb{R}^{2\nu}$
is a diffeomorphic copy of the invariant manifold $W\subset\,\mathbb{R}^{2n}$.
Importantly, $\boldsymbol{\Phi}(W)$ is then an invariant manifold
for the discrete dynamical system (\ref{eq:DISCeom}) by definition.
On this invariant manifold, the map $\boldsymbol{F}$ is conjugate
to the flow map $\boldsymbol{\varXi}_{T}$ by formula (\ref{eq:conjugacy}),
which can now be re-written as 
\begin{equation}
\boldsymbol{F}=\boldsymbol{\Phi}\circ\boldsymbol{\varXi}_{T}\circ\boldsymbol{\Phi}^{-1}:\boldsymbol{\Phi}(W)\to\boldsymbol{\Phi}(W)\label{eq:conjugacy2}
\end{equation}
given that $\boldsymbol{\Phi}$ is a diffeomorphism onto its image. 

Consequently, any coordinate-independent dynamical feature of $\boldsymbol{\varXi}_{T}$
will be shared by the mapping $\boldsymbol{F}$. This will be a crucial
point in our strategy to build a faithful reduced-order model for
system (\ref{eq:starting point}). Specifically, we will use an experimentally
observed scalar $\varphi$ to approximate the Taylor expansion (\ref{eq:expansion of F})
of the mapping $\boldsymbol{F}$. 

Our focus here is the reconstruction of the dynamics of $\boldsymbol{F}$
on two-dimensional invariant manifolds $W$ tangent to two-dimensional
modal subspaces of the linearised flow map $D\boldsymbol{\varXi}_{T}(\boldsymbol{0})$
at the equilibrium point. We thus have $d=2$, and hence the minimal
dimension for the embedding space $\mathbb{R}^{2\nu}$ required by
Takens's theorem is $2\nu\geq5,$ implying $\nu\geq3$ (For our first
example of a two-degree-of-freedom model in section \ref{subsec:Modified-Shaw--Pierre-example},
a comparison with exact analytic computation shows that a reconstruction
with $\nu=2$ already suffices, but this cannot be generally guaranteed.)

The tangent space $T_{\boldsymbol{0}}W$ of $W$ at the origin is
a two-dimensional invariant subspace for $D\boldsymbol{\varXi}_{T}(\boldsymbol{0})$.
Specifically, $T_{\boldsymbol{0}}W$ is the modal subspace corresponding
to a pair of complex conjugate eigenvalues $\left(\mu_{\ell},\bar{\mu}_{\ell}\right)=\left(e^{\lambda_{\ell}T},e^{\bar{\lambda}_{\ell}T}\right)$,
where $(\lambda_{\ell},\bar{\lambda}_{\ell})$ are eigenvalues of
$\boldsymbol{\mathcal{{A}}}$, ordered as in (\ref{eq:assumption on eigenvalues of ODE}).
The conjugacy relationship (\ref{eq:conjugacy2}) and formula (\ref{eq:F(0)=00003D0})
then implies that $\left(\mu_{\ell},\bar{\mu}_{\ell}\right)$ are
also eigenvalues of the linearised sampling map $D\boldsymbol{F}(\boldsymbol{\boldsymbol{0}})$
at $\boldsymbol{\xi}=\boldsymbol{0}$, i.e., we have 
\begin{equation}
\left\{ \mu_{\ell},\bar{\mu}_{\ell}\right\} =\left\{ e^{\lambda_{\ell}T},e^{\bar{\lambda}_{\ell}T}\right\} \subset\mathrm{Spect}\left\{ D\boldsymbol{F}(\boldsymbol{\boldsymbol{0}})\right\} ,\label{eq:spectral containment}
\end{equation}
where $\mathrm{Spect}\left\{ D\boldsymbol{F}(\boldsymbol{\boldsymbol{0}})\right\} $
denotes the spectrum (i.e., set of eigenvalues) of the Jacobian matrix
$D\boldsymbol{F}(\boldsymbol{\boldsymbol{0}})$.

\section{Spectral submanifolds of the sampling map\label{sec:Spectral-submanifolds }}

The linearised sampled dynamics near the fixed point $\boldsymbol{\xi}=\boldsymbol{0}$
is governed by the Jacobian $\boldsymbol{A}=D\boldsymbol{F}(\boldsymbol{0})$
of the sampling map $\boldsymbol{F}$. We assume that this Jacobian
is diagonalisable and collect its complex eigenvectors in a matrix
$\boldsymbol{V}\in\mathbb{C}^{2\nu\times2\nu}$. Introducing the new
coordinate $\boldsymbol{y}\in\mathbb{C}^{2\nu}$ via the relation
\begin{equation}
\boldsymbol{\xi}=\boldsymbol{V}\boldsymbol{y},\label{eq:y definition}
\end{equation}
we obtain the transformed form of (\ref{eq:DISCeom}) as
\begin{equation}
\boldsymbol{y}_{k+1}=\boldsymbol{\Lambda}\boldsymbol{y}_{k}+\boldsymbol{G}(\boldsymbol{y}_{k}),\qquad\boldsymbol{\Lambda}=\mathrm{diag}(\mu_{1},\mu_{2}\ldots,\mu_{2\nu})=\boldsymbol{V}^{-1}\boldsymbol{A}\boldsymbol{V,}\qquad\mu_{2l}=\bar{\mu}_{2l-1},\quad l=1,\ldots,\nu.\label{eq:DISCdiag}
\end{equation}
where $\boldsymbol{G}(\boldsymbol{y})$ are nonlinear coupling terms
with $D\boldsymbol{G}(\boldsymbol{0})=\boldsymbol{0}$. If, specifically,
the $l^{th}$ linear mode of system (\ref{eq:starting point}) is
brought to the standard form
\[
\ddot{\eta}_{l}+2\zeta_{l}\omega_{l}\dot{\eta_{l}}+\omega_{l}^{2}\eta_{l}=0,
\]
with the damping ratio $\zeta_{l}<1$ and undamped natural frequency
$\omega_{l}\in\mathbb{R}$, then we obtain 
\[
\lambda_{l},\bar{\lambda}_{l}=-\zeta_{l}\omega_{l}\pm i\sqrt{1-\zeta_{l}^{2}}\omega_{l},
\]
and hence the eigenvalues of $\boldsymbol{A}$ and $\boldsymbol{\Lambda}$
can be expressed as 
\begin{equation}
\mu_{l},\bar{\mu}_{l}=e^{-\zeta_{l}\omega_{l}T}e^{\pm i\sqrt{1-\zeta_{l}^{2}}\omega_{l}T}.\label{eq:mu_j for proportional damping}
\end{equation}

We recall that if the sampling map $\boldsymbol{F}$ was constructed
from observables along a two-dimensional invariant manifold $W$ of
system (\ref{eq:firstorder}), then $\boldsymbol{\Lambda}$ has a
complex conjugate pair of eigenvalues related to a pair of eigenvalues
of $\boldsymbol{\mathcal{A}}$ through the relationship (\ref{eq:spectral containment}).
In particular, $W$ is tangent to an underdamped modal subspace $E$
of the equilibrium $\boldsymbol{y}=\boldsymbol{0}$ corresponding
to the eigenvalue pair $\mu_{\ell},\overline{\mu}_{\ell}$ for some
for some $\ell\in[1,2n-1],$ as implied by assumption (\ref{eq:assumption on eigenvalues of ODE}). 

The existence of a two-dimensional invariant manifold $W$ tangent
to a two-dimensional spectral subspace $E$ of the linearised system
(\ref{eq:linearized}) was first envisioned in the seminal work of
Shaw and Pierre \citep{ShawPierre}, then extended to more general
settings by the same authors and collaborators (see, e.g., Vakakis
\citep{vakakis2001normal}, Kerschen et al. \citep{kerschen2009nonlinear},
Peeters et al. \citep{Peeters}, and Avramov and Mikhlin \citep{mikhlin2010nonlinears,avramov2013review}
for reviews). On closer inspection, one finds that such invariant
manifolds indeed exist under certain nonresonance conditions, but
are non-unique and may have a low order of differentiability (Neild
et al.\citep{Neild20140404}, Cirillo et al. \citep{Cirillo2015},
Haller and Ponsioen \citep{Haller2016}).

Following Haller and Ponsioen \citep{Haller2016}, we address this
uniqueness and smoothness issue with the help of the following definition:
\begin{defn}
A \emph{spectral submanifold} (SSM) $W(\mathcal{E})$ corresponding
to a spectral subspace $\mathcal{E}$ of the operator $\boldsymbol{A}$
is an invariant manifold of $\boldsymbol{F}$ with the following properties:
\begin{description}
\item [{(i)}] $W(\mathcal{E})$ is tangent to $\mathcal{E}$ at $\boldsymbol{y}=\boldsymbol{0}$
and has the same dimension as $\mathcal{E}$;
\item [{(ii)}] $W(\mathcal{E})$ is strictly smoother than any other invariant
manifold satisfying (i).
\end{description}
\end{defn}

If exists, an SSM serves as the unique nonlinear continuation of the
modal subspace $\mathcal{E}$ to the nonlinear system (\ref{eq:firstorder}).
By definition, all other invariant manifolds tangent to the same modal
subspace have only finitely many derivatives, and hence any high-enough
order Taylor expansion is only valid for the SSM.

Haller and Ponsioen \citep{Haller2016} has pointed out that the more
general and abstract results of Cabré et al. \citep{CabreLlave2003}
imply the existence of spectral submanifolds for $\boldsymbol{F}$
under appropriate conditions on the spectrum of $\boldsymbol{A}$.
Below we recall these results stated specifically in the context of
the sampling map $\boldsymbol{F}$. We note that by the conjugacy
relation (\ref{eq:conjugacy2}), the existence of a two-dimensional
SSM for the sampling map $\boldsymbol{F}$ is equivalent to the existence
of a two-dimensional SSM for the mechanical system (\ref{eq:firstorder}),
as long as the observable $\varphi$ is generically chosen. 

We start by considering a two-dimensional eigenspace $\mathcal{E}\subset\mathbb{C}^{2\nu}$
of the linearised sampling map $\boldsymbol{A}$, corresponding to
the eigenvalue pair $\mu_{\ell},\overline{\mu}_{\ell}$. We define
the \emph{relative spectral quotient $\sigma(\mathcal{E})$ }as the
positive integer\foreignlanguage{english}{ }
\begin{eqnarray}
\sigma(\mathcal{E}) & = & \mathrm{Int}\,\left[\frac{\underset{j\neq\ell,\ell+1}{\min}\log\left|\mu_{j}\right|}{\log\left|\mu_{\ell}\right|}\right]\in\mathbb{N}^{+}.\label{eq:relative_sigma}
\end{eqnarray}
For the linearised sampling map $\boldsymbol{A}$ (\ref{eq:linearized}),
the constant $\sigma(\mathcal{E})$ is the integer part of the ratio
of the strongest decay rate towards the spectral subspace \foreignlanguage{english}{$\mathcal{E}$}
to the decay rate along $\mathcal{E}$. This integer ratio turns out
to control the smoothness of the spectral submanifold $W\left(\mathcal{E}\right)$,
as we shall see shortly.

We assume now that
\begin{equation}
\sigma(\mathcal{E})\leq r,\label{eq:spectral condition}
\end{equation}
i.e., the degree of smoothness of the sampling map $\boldsymbol{F}$
is at least as high as the relative spectral quotient of the modal
subspace $\mathcal{E}$. Finally, we assume that no resonance relationships
between order $2$ and order $\sigma(\mathcal{E})$ hold between the
eigenvalues $\mu_{\ell},\overline{\mu}_{\ell}$ and the rest of the
spectrum of $\boldsymbol{A}$:

\begin{equation}
\mu_{\ell}^{s_{1}}\bar{\mu}_{\ell}^{s_{2}}\neq\mu_{j},\qquad\forall j\neq\ell,\ell+1,\qquad2\le s_{1}+s_{2}\le\sigma\left(\mathcal{E}\right).\label{eq:nonresonance}
\end{equation}
We then have the following existence and uniqueness result for spectral
submanifolds of the sampling map.
\begin{thm}
\label{thm: existence SSM}Assume that conditions (\ref{eq:spectral condition})-(\ref{eq:nonresonance})
are satisfied. Then the following statements hold:
\end{thm}
\begin{description}
\item [{(i)}] There exists an SSM, $W\left(\mathcal{E}\right),$ for the
nonlinear sampling map $\boldsymbol{F}$, that is tangent to the invariant
subspace $\mathcal{E}$ at the $\boldsymbol{\xi}=\boldsymbol{0}$
fixed point. 
\item [{(ii)}] The invariant manifold $W\left(\mathcal{E}\right)$ is class
$C^{r}$ smooth and unique among all two-dimensional, class $C^{\sigma\left(\mathcal{E}\right)+1}$
invariant manifolds of $\boldsymbol{F}$ that are tangent to $\mathcal{E}$
at $\boldsymbol{\xi}=\boldsymbol{0}$. 
\item [{(iii)}] The SSM $W\left(\mathcal{E}\right)$ can be viewed as a
$C^{r}$ immersion of an open set $\mathcal{U}\subset\mathbb{C}^{2}$
into the phase space $\mathbb{C}^{2\nu}$ of $\boldsymbol{F}$ via
a map 
\begin{eqnarray}
\boldsymbol{W}:\mathcal{U}\subset\mathbb{C}^{2} & \to & \mathbb{C}^{2\nu},\qquad\boldsymbol{W}\left(\mathcal{U}\right)=W\left(\mathcal{E}\right).\label{eq:W embedding}
\end{eqnarray}
\item [{(iv)}] There exists a $C^{r}$ polynomial map $\boldsymbol{R}\colon\mathcal{U}\to\mathcal{U}$
such that 
\begin{equation}
\boldsymbol{F}\circ\boldsymbol{W}=\boldsymbol{W}\circ\boldsymbol{R},\label{eq:invariancecond0}
\end{equation}
 i.e., the dynamics on the SSM, expressed in the coordinates $\boldsymbol{z}=(z_{\ell},\bar{z}_{\ell})\in\mathcal{U},$
is given by the polynomial mapping $\boldsymbol{R}$. This polynomial
mapping only has terms up to order $\mathcal{O}\left(\left|\boldsymbol{z}\right|^{\sigma\left(\mathcal{E}\right)}\right)$.
\item [{(v)}] If, for some integer $j_{0}\geq2$, all internal non-resonance
conditions 
\begin{equation}
\mu_{\ell}^{s_{1}}\bar{\mu}_{\ell}^{s_{2}}\neq\mu_{\ell},\mu_{\ell+1},\qquad j_{0}\le s_{1}+s_{2}\le\sigma\left(\mathcal{E}\right)\label{eq:strict nonresonance}
\end{equation}
hold\emph{ within} $\mathcal{E}$, then the polynomial $\boldsymbol{R}$
in (\ref{eq:invariancecond0}) can be selected to contain only terms
up to order $j_{0}-1.$
\item [{(vi)}] If the observable $\varphi$ used in the construction of
the sampling map $\boldsymbol{F}$ is generic, then a two-dimensional
SSM, $W\left(E\right),$ tangent to the subspace $E$ at $\boldsymbol{x}=\boldsymbol{0}$
exists for the original system (\ref{eq:firstorder}). The invariant
manifold $W\left(E\right)$ shares the properties (i)-(v) of $W\left(\mathcal{E}\right)$
due to the conjugacy relationship (\ref{eq:conjugacy2}).
\end{description}
\begin{proof}
As explained in detail by Haller and Ponsioen \citep{Haller2016},
the proofs of statements (i)-(v) follow from a direct application
of the more general Theorem 1.1 of Cabré et al. \citep{CabreLlave2003}
on invariant manifolds tangent to spectral subspaces of arbitrary
dimension, for mappings defined on Banach spaces. Statement (vi) can
be concluded by invoking the strengthened version of Taken's theorem
(Huke \citep{huke2006embedding}, Stark \citep{Stark1999}) that we
recalled in section \ref{sec:Delay-embedding}, then interpreting
the resulting structures of the Poincaré map $\boldsymbol{\varXi}_{T}$
for the flow map of (\ref{eq:firstorder}).
\end{proof}

\begin{rem}
\label{rem: case of proportional damping}If the linearisation of
the original mechanical system (\ref{eq:starting point}) satisfies
the classic proportional damping hypothesis, we can use (\ref{eq:mu_j for proportional damping})
to rewrite the relative spectral quotient defined in (\ref{eq:relative_sigma})
as 
\begin{equation}
\sigma(\mathcal{E})=\mathrm{Int}\,\left[\frac{\underset{j\neq\ell,\ell+1}{\max}\zeta_{j}\omega_{j}}{\zeta_{\ell}\omega_{\ell}}\right].\label{eq:sigma from parameters}
\end{equation}
 In this case, the expressions in the external nonresonance conditions
(\ref{eq:nonresonance}) for $\mathcal{E}$ take the specific form
\[
\left(\begin{array}{r}
\left(s_{1}+s_{2}\right)\zeta_{\ell}\omega_{\ell}\\
\left(s_{1}-s_{2}\right)\omega_{\ell}\sqrt{1-\zeta_{\ell}^{2}}\;\mod\;\nicefrac{2\pi}{T}
\end{array}\right)\neq\left(\begin{array}{l}
\zeta_{j}\omega_{j}\\
\omega_{nj}\sqrt{1-\zeta_{j}^{2}}\;\mod\;\nicefrac{2\pi}{T}
\end{array}\right),\qquad2\le s_{1}+s_{2}\le\sigma\left(\mathcal{E}\right),
\]
where $\mod$ denotes the modulo operation that takes sampling into
account. In the limit of zero damping, an external resonance for $\mathcal{E}$
means that a frequency $\omega_{j}$ outside $\mathcal{E}$ is an
integer multiple of the frequency $\omega_{\ell}$ inside $\mathcal{E}$. 
\end{rem}

Statements (iii)-(v) of Theorem \ref{thm: existence SSM} imply that,
unlike in the Shaw-Pierre \citep{ShawPierre} construction, the SSM
inferred from the results of Cabré et al. \citep{CabreLlave2003}
is not assumed to be a graph over the subspace $\mathcal{E}$ in the
phase space of $\boldsymbol{F}$. This allows $W\left(\mathcal{E}\right)$
to be constructed on larger domains on which it can produce folds
over $\mathcal{E}$. This parametrisation approach to SSM construction
was also re-discovered recently by Cirillo et al. \textbf{\citep{Cirillo2015}
}under the assumption that the flow is analytically linearisable near
the fixed point $\boldsymbol{x}=\boldsymbol{0}$. Analytic linearisation
does not allow for any resonance in the spectrum of $\boldsymbol{A}$
and, in return, transforms the full dynamics of the mapping $\boldsymbol{F}$
into that of $\boldsymbol{\Lambda y}$. In the case of a near-resonance
\textendash{} which arises for all weakly underdamped modes, as we
shall see below \textendash{} analytic linearisation can therefore
only be constructed on a very small domain near the fixed point. This
disallows the type of direct identification of nonlinear terms that
we discuss next.

\section{Dynamics on SSMs: Backbone curves\label{sec:backbone curves}}

Since $\left|\mu_{\ell}\right|<1$ holds by assumption (\ref{eq:assumption on eigenvalues of ODE}),
we find that, strictly speaking, the internal non-resonance condition
(\ref{eq:strict nonresonance}) is always satisfied for nonzero damping.

As seen in the construct of Cabré et al. \citep{CabreLlave2003},
however, even an approximate resonance $\mu_{\ell}^{s_{1}}\bar{\mu}_{\ell}^{s_{2}}\approx\mu_{j}$
causes the near-identity transformation $(\xi_{\ell},\bar{\xi}_{\ell})\mapsto(z_{\ell},\bar{z}_{\ell})$
to have small denominators, limiting the existence of this transformation
to a tiny neighbourhood of the $\boldsymbol{\xi}=\boldsymbol{0}$
fixed point. Since our interest here is to obtain an approximation
of the dynamics of $\boldsymbol{F}$ on a sizeable neighbourhood of
the fixed point within the SSM, we do not insist on the removal of
approximately resonant terms in the $(z_{\ell},\bar{z}_{\ell})$ coordinate
system. Rather, we observe that for small damping ratios (i.e., for
$\left|\mu_{\ell}\right|\approx1)$, the low-order near-resonance
relationships 
\begin{equation}
\mu_{\ell}^{2}\bar{\mu}_{\ell}\approx\mu_{\ell},\quad\mu_{\ell}\bar{\mu}_{\ell}^{2}\approx\bar{\mu}_{\ell}\label{eq:near resonance}
\end{equation}
are always satisfied, and hence the minimal possible integer $j_{0}$
satisfying (\ref{eq:strict nonresonance}) (with $\neq$ replaced
with $\not\approx$) is $j_{0}=1$. Accordingly, the approximately
failing resonance conditions in (\ref{eq:strict nonresonance}) prompt
us to seek $\boldsymbol{R}$ (cf. statement (iv) of Theorem \ref{thm: existence SSM})
as a cubic polynomial of the form 
\begin{align}
\boldsymbol{R}(\boldsymbol{z}) & =\begin{pmatrix}\begin{array}{l}
\mu_{\ell}z_{\ell}+\beta_{\ell}z_{\ell}^{2}\bar{z}_{\ell}+\ldots\\
\bar{\mu}_{\ell}\bar{z}_{\ell}+\bar{\beta_{\ell}}z_{\ell}\bar{z}_{\ell}^{2}+\ldots
\end{array}\end{pmatrix}.\label{eq:1DOFred}
\end{align}

Introducing polar coordinates $z=\rho e^{i\theta}$, we can further
transform (\ref{eq:1DOFred}) to the real amplitude-phase components
\begin{eqnarray}
\rho_{\ell} & \mapsto & \rho_{\ell}\left|\mu_{\ell}+\beta_{\ell}\rho_{\ell}^{2}\right|,\label{eq:RadAmpl}\\
\theta_{\ell} & \mapsto & \theta+\arg\left(\mu_{\ell}+\beta_{\ell}\rho_{\ell}^{2}\right).\label{eq:RadAngle}
\end{eqnarray}
Equation (\ref{eq:RadAngle}) then provides an instantaneous frequency
of nonlinear oscillations, with the instantaneous oscillation amplitude
governed by equation (\ref{eq:RadAmpl}). Given that the sampling
period we have used is $T$, the leading-order approximation of the
instantaneous oscillation frequency in the original nonlinear system
(\ref{eq:firstorder}) is 
\begin{equation}
\omega(\rho_{\ell})=\frac{\arg\left(\mu_{\ell}+\beta_{\ell}\rho_{\ell}^{2}\right)}{T}.\label{eq:omegaOFrho}
\end{equation}

We take the instantaneous leading-order amplitude of the corresponding
trajectories of (\ref{eq:firstorder}) to be the norm of $\boldsymbol{z}(\rho_{\ell},\theta_{\ell})=\left(\rho_{\ell}e^{i\theta_{\ell}},\rho_{\ell}e^{-i\theta_{\ell}}\right)$
in the original $\boldsymbol{\xi}$ coordinates. A nominal instantaneous
amplitude $\mathrm{Amp}(\rho)$ of the vibration can then be calculated
from (\ref{eq:y definition}) as the $L_{2}$ norm of the norm of
$\boldsymbol{\boldsymbol{z}}(\rho_{\ell},\theta_{\ell})$ in the original
$\boldsymbol{\xi}$ coordinates: 
\begin{equation}
\mathrm{Amp}(\rho_{\ell})=\sqrt{\frac{1}{2\pi}\intop_{0}^{2\pi}\left|\boldsymbol{V}\boldsymbol{\,W}(\boldsymbol{z}(\rho_{\ell},\theta_{\ell}))\right|^{2}\,d\theta.}\label{eq:closed curve}
\end{equation}
Here, the linear map $\boldsymbol{V}$ is the one appearing in (\ref{eq:y definition}),
and the mapping $\boldsymbol{W}$ in the one appearing in (\ref{eq:W embedding}). 
\begin{defn}
\label{def:backbone}We call the parametrised curve
\begin{equation}
\mathcal{B}_{\ell}=\left\{ \omega(\rho_{\ell}),\mathrm{Amp}(\rho_{\ell})\right\} _{\rho_{\ell}\in\mathbb{R}^{+}}\subset\mathbb{R}^{2}\label{eq:BackGraph}
\end{equation}
the \emph{backbone curve }associated with the nonlinear dynamics on
the SSM, $W(\mathcal{\mathcal{E}})$. 
\end{defn}
The key to the computation of the backbone curve (\ref{eq:BackGraph})
is, therefore, the computation of the single complex coefficient $\beta_{\ell}$
and of the mapping $\boldsymbol{W}(\boldsymbol{z})$. This is because
both the eigenvalue $\mu_{\ell}$ and the sampling time $T$ are already
assumed to be known. 
\begin{rem}
\label{rem: observed backbone}It is often desirable to translate
the $\varphi$-based backbone curve $\mathcal{B}_{\ell}$ defined
in (\ref{eq:BackGraph}) to a backbone curve observed directly for
a given mechanical coordinate $q_{j}$. When the observable is an
invertible function of such a $q_{j}$, that is, $\varphi(\boldsymbol{q},\dot{\boldsymbol{q}})=\varphi(q_{j})$,
we can use the inverse, defined by $q_{j}=\mathcal{P}(\mathbf{\boldsymbol{\xi}})=\mathcal{P}(\varphi(q_{j}))$.
Also notice that by the definition of the observable space, coordinates
of $\boldsymbol{\xi}$ are just sampled values of the same observed
quantity. Therefore, when calculating an amplitude, it is reasonable
to consider just a single component of $\boldsymbol{\xi}$, for example,
$\xi_{1}$. With this in mind, we consider $\mathcal{P}$ functions
in the particular form $\mathcal{P}(\boldsymbol{\xi})=\mathcal{P}(\xi_{1})$.
As a result, the observed amplitude in the $q_{j}$ mechanical coordinate
can be computed as 
\[
\mathrm{Amp}(\rho_{\ell})=\sqrt{\frac{1}{2\pi}\intop_{0}^{2\pi}\left|\mathcal{P}\left(\boldsymbol{V}\boldsymbol{\,W}(\boldsymbol{z}(\rho_{\ell},\theta_{\ell}))\right)\right|^{2}\,d\theta}.
\]
\end{rem}
To compute the complex parameter $\beta$ in equation (\ref{eq:omegaOFrho}),
we need to solve (\ref{eq:invariancecond0}). To this end, we seek
the Taylor series coefficients of the $j^{th}$ coordinate functions,
$W_{j}(\boldsymbol{z})\in\mathbb{C}$, $j=1,\ldots,2\nu,$ of the
mapping $\boldsymbol{W}(\boldsymbol{z})$ up to third order. Similarly,
we week the third-order Taylor coefficient $\beta_{\ell}\in\mathbb{C}$
of the polynomial mapping $\boldsymbol{R}(\boldsymbol{z})$ defined
in (\ref{eq:1DOFred}). All these unknowns should be expressed in
the end as functions of the $j^{th}$ coordinate functions $G_{j}(\boldsymbol{y})\in\mathbb{C},$
$j=1,\ldots,2\nu,$ of the nonlinear part $\boldsymbol{G}(\boldsymbol{y})$
of the transformed sampling map $\boldsymbol{F}$. The relevant Taylor
expansions are in the general form 
\begin{align}
G_{j}(\boldsymbol{y}) & =\sum_{\left|\boldsymbol{m}\right|\geq2}g_{j}^{\boldsymbol{m}}\boldsymbol{y}^{\boldsymbol{m}},\;\boldsymbol{m}\in\mathbb{N}^{2\nu},\quad g_{j}^{\boldsymbol{m}}\in\mathbb{C},\qquad j=1,\ldots,2\nu,\label{eq:G expand}\\
W_{j}(\boldsymbol{z}) & =\sum_{\left|\boldsymbol{s}\right|\geq1}w_{j}^{\boldsymbol{s}}\boldsymbol{z}^{\boldsymbol{s}},\;\boldsymbol{s}\in\mathbb{N}^{2},\quad\quad w_{j}^{\boldsymbol{s}}\in\mathbb{C},\qquad j=1,\ldots,2\nu.\label{eq:W expand}
\end{align}

In expressing the solutions of (\ref{eq:invariancecond0}) in terms
of these coefficients, we will use the short-hand notation $(p@j)$
for an integer multi-index whose elements are zero, except for the
one at the $j^{th}$ position, which is equal to $p$:
\[
(p@j):=\left(0,\ldots,\underset{j-1}{0},\underset{j}{p},\underset{j+1}{0},\ldots,0\right)\in\mathbb{N}^{2\nu}.
\]
We will also concatenate this notation to refer to multi-indices whose
entries are zero except at prescribed locations:
\[
(p@j_{1},q@j_{2}):=\left(\underset{}{0},\ldots,\underset{j_{1}-1}{0},\underset{j_{1}}{p},\underset{j_{1}+1}{0},\underset{\cdots}{\ldots},\underset{j_{2}-1}{0},\underset{j_{2}}{q},\underset{j_{2}+1}{0},\underset{\cdots}{\ldots},\underset{}{0}\right)\in\mathbb{N}^{2\nu}.
\]
For $j_{1}\equiv j_{2}=j,$ we let
\[
(p@j,q@j):=(\left(p+q\right)@j)=\left(\ldots,\underset{j-1}{0},\underset{j}{p+q},\underset{j+1}{0},\underset{\cdots}{\ldots}\right)\in\mathbb{N}^{2\nu}.
\]

With all this notation, we obtain the following result:
\begin{thm}
\label{thm:coefficients} Suppose that the assumptions of Theorem
\ref{thm: existence SSM} hold but with the strengthened version 
\begin{equation}
\mu_{\ell}^{s_{1}}\bar{\mu}_{\ell}^{s_{2}}\not\approx\mu_{j},\qquad\forall j\neq\ell,\ell+1,\qquad1\le s_{1}+s_{2}\le\sigma\left(\mathcal{E}\right)\label{eq:nonresonance-2}
\end{equation}
of the external non-resonance condition (\ref{eq:nonresonance}).
Then, for any $j\in[1,2\nu]$, the $j^{th}$ coordinate function $W_{j}$
of the mapping $\boldsymbol{W}$ and the cubic Taylor coefficient
$\beta_{\ell}$ of the conjugate map $\boldsymbol{R}$ are given by
the following formulas:
\[
w_{j}^{(1,0)}=\delta_{j\ell},\qquad w_{j}^{(0,1)}=\delta_{j(\ell+1)},
\]
\[
w_{j}^{(2,0)}=\frac{g_{j}^{\left(2@\ell\right)}}{\mu_{\ell}^{2}-\mu_{j}},\qquad w_{j}^{(1,1)}=\frac{g_{j}^{\left(1@\ell,1@\left(\ell+1\right)\right)}}{\mu_{\ell}\overline{\mu}_{\ell}-\mu_{j}},\qquad w_{j}^{(0,2)}=\frac{g_{j}^{\left(2@\left(\ell+1\right)\right)}}{\overline{\mu}_{\ell}^{2}-\mu_{j}},
\]
\[
w_{j}^{(3,0)}=\frac{\sum_{q=1}^{2\nu}\left(1+\delta_{\ell q}\right)g_{j}^{(1@\ell,1@q)}w_{q}^{(2,0)}+g_{j}^{(3@\ell)}}{\mu_{\ell}^{3}-\mu_{j}},\qquad w_{j}^{(0,3)}=\frac{\sum_{q=1}^{2\nu}\left(1+\delta_{\left(\ell+1\right)q}\right)g_{j}^{(1@(\ell+1),1@q)}w_{q}^{(0,2)}+g_{j}^{(3@\left(\ell+1\right))}}{\bar{\mu}_{\ell}^{3}-\mu_{j}}.
\]
 
\[
w_{j}^{(2,1)}=\left(1-\delta_{j\ell}\right)\frac{\sum_{q=1}^{2\nu}\left[\left(1+\delta_{\ell q}\right)g_{j}^{(1@\ell,1@q)}w_{q}^{(1,1)}+\left(1+\delta_{\left(\ell+1\right)q}\right)g_{j}^{(1@(\ell+1),1@q)}w_{q}^{(2,0)}\right]+g_{j}^{(2@\ell,1@(\ell+1))}}{\mu_{\ell}^{2}\bar{\mu}_{\ell}-\mu_{j}},
\]
\[
w_{j}^{(1,2)}=\left(1-\delta_{j(\ell+1)}\right)\frac{\sum_{q=1}^{2\nu}\left[\left(1+\delta_{\ell q}\right)g_{j}^{(1@\ell,1@q)}w_{q}^{(0,2)}+\left(1+\delta_{\left(\ell+1\right)q}\right)g_{j}^{(1@(\ell+1),1@q)}w_{q}^{(1,1)}\right]+g_{j}^{(2@(\ell+1),1@\ell)}}{\mu_{\ell}\bar{\mu}_{\ell}^{2}-\mu_{j}},
\]
 
\[
\beta_{\ell}=\sum_{q=1}^{2\nu}\left[\left(1+\delta_{\ell q}\right)g_{\ell}^{(1@\ell,1@q)}w_{q}^{(1,1)}+\left(1+\delta_{\left(\ell+1\right)q}\right)g_{\ell}^{(1@(\ell+1),1@q)}w_{q}^{(2,0)}\right]+g_{\ell}^{(2@\ell,1@(\ell+1))}.
\]
\end{thm}
\begin{proof}
See Appendix A.
\end{proof}

\begin{rem}
\label{rem:higher order} Theorem \ref{thm:coefficients} only provides
the solution of the homological equation (\ref{eq:invariancecond0})
up to cubic order. This equation, however, can be solved by symbolic
computations up to any order for the Taylor coefficients of the functions
$\boldsymbol{W}$ and $\boldsymbol{R}$. For instance, up to quintic
order, the near-resonance conditions (\ref{eq:near resonance}) imply
the general form 
\[
\boldsymbol{R}(\boldsymbol{z})=\begin{pmatrix}\begin{array}{l}
\mu_{\ell}z_{\ell}+\beta_{\ell}z_{\ell}^{2}\bar{z}_{\ell}+\gamma_{\ell}z_{\ell}^{3}\bar{z}_{\ell}^{2}+\ldots\\
\bar{\mu}_{\ell}\bar{z_{\ell}}+\bar{\beta_{\ell}}z_{\ell}\bar{z}_{\ell}^{2}+\bar{\gamma}_{\ell}z_{\ell}^{2}\bar{z}_{\ell}^{3}+\ldots
\end{array}\end{pmatrix}
\]
for the polynomial conjugate dynamics on the SSM $\mathcal{E}$. The
coefficient $\gamma_{\ell}$ as well as the quartic and quintic terms
of $\boldsymbol{W}$ can be found recursively from equation (\ref{eq:invariancecond0}),
following the procedure outlined in Appendix A. The sampling map restricted
to the SSM $W\left(\mathcal{E}\right)$ can be written in polar coordinates
up to quintic order as 
\begin{eqnarray}
\rho_{\ell} & \mapsto & \rho_{\ell}\left|\mu_{\ell}+\beta_{\ell}\rho_{\ell}^{2}+\gamma_{\ell}\rho_{\ell}^{4}\right|,\label{eq:RadAmpl-2}\\
\theta_{\ell} & \mapsto & \theta+\arg\left(\mu_{\ell}+\beta_{\ell}\rho_{\ell}^{2}+\gamma_{\ell}\rho_{\ell}^{4}\right),\label{eq:RadAngle-2}
\end{eqnarray}
yielding the instantaneous oscillation frequency in the original nonlinear
system (\ref{eq:firstorder}) as 
\begin{equation}
\omega(\rho_{\ell})=\frac{\arg\left(\mu_{\ell}+\beta_{\ell}\rho_{\ell}^{2}+\gamma_{\ell}\rho_{\ell}^{4}\right)}{T}.\label{eq:quintic freq}
\end{equation}
The formulas (\ref{eq:closed curve}) and (\ref{eq:quintic freq})
then give a refined, quintic approximation for the backbone curve
$\mathcal{B}_{\ell}$. The same procedure applies to further, higher-order
approximations of $\mathcal{B}_{\ell}$. 
\end{rem}

\begin{rem}
\label{rem: coeffs for strong resonance}The external nonresonance
condition (\ref{eq:nonresonance}) of Theorem \ref{thm:coefficients}
only excludes quadratic and higher-order resonances. As a result,
for overdamped spectral submanifolds $\mathcal{E}$ with eigenvalues
$\mu_{\ell},\mu_{\ell+1}\in\mathbb{R}$, condition (\ref{eq:nonresonance})
would still technically allow for a $1:1$ external resonance (characterised
by $s_{1}=1$ and $s_{2}=0$ ) with an eigenvalue $\mu_{j}\in\mathbb{R}$
outside $\mathcal{E}$ . In our setting, however, the damping is assumed
weak and hence an approximate $1:1$ external resonance $\mu_{\ell}\approx\mu_{j}$
implies an approximate external $2:1$ resonance $\mu_{\ell}^{2}\mu_{\ell+1}^{1}\approx\mu_{j}$,
resulting in small denominators for $w_{j}^{(1,2)}$ and $w_{j}^{(2,1)}$
in the statement of Theorem \ref{thm:coefficients}. The strengthened
nonresonance condition (\ref{eq:nonresonance-2}) serves to exclude
this case, as well as other cases of near-resonance that create nonzero
but small denominators for the coefficients in Theorem \ref{thm:coefficients}.
Although technically nonzero, such small denominators are undesirable
as they may significantly decrease the phase space domain on which
the formulas of the theorem give a good approximation for the underlying
SSM and its reduced dynamics. 
\end{rem}

\section{Reconstruction of the sampling map from data \label{subsec:ModelID}}

In an experimental setting, backbone-curve identification via Theorem
\ref{thm:coefficients} requires the fitting of a model of $\boldsymbol{F}$
to observations using an appropriate set of basis functions. Due to
the polynomial form (\ref{eq:expansion of F}) of $\boldsymbol{F}$,
the required basis functions are precisely vector-valued monomials
of the variables $\xi_{1},\ldots,\xi_{2\nu}$ not including constant
terms. The lack of constant terms follows from the assumption (\ref{eq:F(0)=00003D0}),
which can always be satisfied by an appropriate shift of coordinates,
if necessary.

For the polynomial-based model-identification for $\boldsymbol{F}$,
we employ a nonlinear autoregressive model (NAR) (Billings \citep{billings2013}).
We order all integer vectors $\boldsymbol{m}$ up to order $\left|\boldsymbol{m}\right|=r$
(i.e., all index vectors in the leading-order Taylor expansion (\ref{eq:expansion of F}))
into a series $\left\{ \boldsymbol{m}^{l}\right\} $ so that 
\[
\boldsymbol{m}^{v}\prec\boldsymbol{m}^{w}\quad\Longleftrightarrow\quad m_{j}^{v}\leq m_{j}^{w},\quad j=1,\ldots,2\nu.
\]
We can then write the yet unknown, $r^{th}$-order Taylor expansion
of $\boldsymbol{F}$ in the compact form 
\begin{equation}
\boldsymbol{F}(\boldsymbol{\xi})=\boldsymbol{K}\boldsymbol{\psi}(\boldsymbol{\xi})+\boldsymbol{r}(\boldsymbol{\xi}),\qquad\psi_{l}(\boldsymbol{\xi})=\boldsymbol{\xi}^{\boldsymbol{m}^{l}},\label{eq:FitMap}
\end{equation}
where $\boldsymbol{K}\in\mbox{\ensuremath{\mathbb{R}}}^{2\nu\times N}$
is a rectangular matrix, to be determined by minimising the residual
term $\boldsymbol{r}(\boldsymbol{\xi})\in\mbox{\ensuremath{\mathbb{R}}}^{2\nu}$
on assimilated data in the $\ell^{2}$ norm.

The input data to be assimilated into the NAR model consists of $P$
sequences of $M_{p}>2\nu$-long observations, $\left\{ \boldsymbol{\xi}_{k}^{p}\right\} _{k=0}^{M_{p}-2\nu}$
, $p=1,\ldots,P,$ with each observation sequence $\left\{ \boldsymbol{\xi}_{k}^{p}\right\} _{k=0}^{M_{p}-2\nu}$
defined as in (\ref{eq:defxi}). The $\ell^{2}$ norm of $\boldsymbol{r}(\boldsymbol{\xi})$
on $\left\{ \boldsymbol{\xi}_{k}^{p}\right\} _{k=0}^{M_{p}-2\nu}$
over all $P$ observation sequences is then given by 
\begin{align*}
Err=\sum_{p=1}^{P}\sum_{k=0}^{M_{p}-2\nu}\left|\boldsymbol{r}\left(\boldsymbol{\xi}_{k}^{p}\right)\right|^{2} & =\sum_{p=1}^{P}\sum_{k=0}^{M_{p}-2\nu}\left|\boldsymbol{K}\boldsymbol{\psi}(\boldsymbol{\xi}_{k}^{p})-\boldsymbol{\xi}_{k+1}^{p}\right|^{2}.
\end{align*}
The matrix $\boldsymbol{K}$ that minimises this norm is obtained
by solving the equation $\mathrm{d}Err/\mathrm{d}\boldsymbol{K}=\boldsymbol{0}$
for $\boldsymbol{K}$. This classic computation yields $\boldsymbol{K}=\boldsymbol{Q}\boldsymbol{P}^{-1}$,
where
\begin{align*}
\boldsymbol{P} & =\sum_{p=1}^{P}M_{p}^{-1}\sum_{k=0}^{M_{p}-2\nu}\boldsymbol{\psi}(\boldsymbol{\xi}_{k}^{p})\boldsymbol{\psi}^{\star}(\boldsymbol{\xi}_{k}^{p}),\\
\boldsymbol{Q} & =\sum_{p=1}^{P}M_{p}^{-1}\sum_{k=0}^{M_{p}-2\nu}\boldsymbol{\xi}_{k+1}^{p}\boldsymbol{\psi}^{\star}(\boldsymbol{\xi}_{k}^{p}),
\end{align*}
with $\star$ denoting the transposition. With this notation, the
reconstructed nonlinear sampling map is 
\begin{equation}
\boldsymbol{\tilde{F}}(\boldsymbol{\xi})=\boldsymbol{\boldsymbol{Q}\boldsymbol{P}^{-1}}\boldsymbol{\psi}(\boldsymbol{\xi}),\label{eq:Ftilde}
\end{equation}
which we will use instead of the exact sampling map $\boldsymbol{F}$
in our analysis. 

Assimilating multiple measurement sequences (i.e., using $P>1$) generally
reduces the effect of zero-mean additive noise on the model reconstruction.
More importantly, using measurements from vibrations decaying near
$P$ natural frequencies of interest allows us to build a single reduced-order
discrete model map $\boldsymbol{\tilde{F}}$ that simultaneously captures
nonlinear behaviour near all these natural frequencies. The choice
of the $\ell^{2}$ optimisation above was mostly dictated by convenience;
in some situations, minimisation of $\boldsymbol{r}(\boldsymbol{\xi})$
in the $\ell^{1}$ or $\ell^{\infty}$ norms might be more beneficial. 

Since we do not know the invariant manifold $W\left(\mathcal{E}\right)$
exactly, we will construct (\ref{eq:Ftilde}) from observed nonlinear
vibration decay measurements initiated along two-dimensional modal
subspaces of $D\boldsymbol{\varXi}_{T}(\boldsymbol{0})$. In practice,
these subspaces can be approximated from linear modal analysis.

\section{Summary of SMM-based backbone-curve identification algorithm\label{sec:Summary-of-algortihm}}

We now briefly summarise the steps in the approach we have developed
in the preceding sections:
\begin{description}
\item [{1.}] Fix a generic scalar observable $\varphi(\boldsymbol{q},\dot{\boldsymbol{q}})$
and a sampling time $T>0$ for the mechanical system (\ref{eq:starting point}).
Also fix an integer $\nu\geq3$ as the number of SSMs to be identified
for system (\ref{eq:starting point}). Finally, select an integer
$r=\max\left|\boldsymbol{m}\right|$ for the maximum degree of the
polynomials used in the construction of the NAR model (\ref{eq:FitMap})
for the sampling map $\boldsymbol{\tilde{F}}(\boldsymbol{\xi})$ with
$\boldsymbol{\xi}\in\mathbb{R}^{2\nu}$. 
\item [{2.}] Collect $P$ sequences of $M_{p}$-long observations, $\left\{ \boldsymbol{\xi}_{k}^{p}\right\} _{k=0}^{M_{p}-2\nu}$,
by letting
\begin{eqnarray*}
\boldsymbol{\xi}_{k}^{p} & = & \left(\varphi\left(\boldsymbol{q}\left(kT\right),\dot{\boldsymbol{q}}\left(kT\right)\right),\ldots,\varphi\left(\boldsymbol{q}\left((k+2\nu-1)T\right),\dot{\boldsymbol{q}}\left((k+2\nu-1)T\right)\right)\right),\\
 &  & p=1,\ldots,P,\qquad k=0,\ldots,M_{p}-2\nu.
\end{eqnarray*}
\item [{3.}] Compute the approximate $2\nu$-dimensional sampling map $\boldsymbol{\tilde{F}}(\boldsymbol{\xi})$
from formula (\ref{eq:Ftilde}).
\item [{4.}] Transform $\boldsymbol{\tilde{F}}(\boldsymbol{\xi})$ to its
complex diagonal form (\ref{eq:DISCdiag}).
\item [{5.}] Using Theorem 1, compute the leading order Taylor coefficients
of the mapping $\boldsymbol{W}(\boldsymbol{\boldsymbol{z}}_{\ell})$
and the leading order polynomial coefficient $\beta_{\ell}$ for each
SSM, $W\text{\ensuremath{\left(\mathcal{E}\right)} }$, provided that
the nonresonance condition (\ref{eq:nonresonance}) holds. 
\item [{6.}] Calculate the backbone curve $\mathcal{B}_{\ell}$ defined
in (\ref{eq:BackGraph}) for $W\text{\ensuremath{\left(\mathcal{E}\right)} }$.
Higher-order approximations to $\mathcal{B}_{\ell}$ can be computed
similarly, as summarised briefly in Remark \ref{rem:higher order}.
\end{description}
This algorithm provides the simplest possible first-order approach
to SSM-based backbone curve reconstruction. This simplest approach
does not fully exploit the uniqueness class $C^{\sigma(\mathcal{E})+1}$
of $W\text{\ensuremath{\left(\mathcal{E}\right)} }$, as guaranteed
by Theorem 1. To obtain higher precision approximations to $\mathcal{B}_{\ell}$,
one must derive higher-order Taylor coefficients of $\boldsymbol{W}(\boldsymbol{\boldsymbol{z}}_{\ell})$
and $\beta_{\ell}$ from the invariance condition (\ref{eq:invariancecond0}),
which we do not pursue here. 

\section{Examples}

We now demonstrate the application of SSM-based model reduction and
backbone-curve reconstruction in two examples. First, we consider
a two-degree-of-freedom damped, nonlinear oscillator model to benchmark
data-based SSM reconstruction in a case where analytic, model-based
computations are also possible. Second, we use vibration decay data
from an oscillating beam experiment to illustrate the direct computation
of backbone curves $\mathcal{B}_{\ell}$ from an experimentally reconstructed
sampling map $\boldsymbol{\tilde{F}}.$ 

\subsection{Modified Shaw\textendash Pierre example\label{subsec:Modified-Shaw--Pierre-example}}

We slightly modify here the two-degree-of-freedom oscillator studied
by Shaw and Pierre \citep{ShawPierre} by making the damping matrix
proportional to the stiffness matrix in the linearised problem. The
first-order equations of motion we study are 
\begin{equation}
\begin{array}{rl}
\dot{x}_{1} & =v_{1},\\
\dot{x}_{2} & =v_{2},\\
\dot{v}_{1} & =-cv_{1}-k_{0}x_{1}-\kappa x_{1}^{3}-k_{0}(x_{1}-x_{2})-c(v_{1}-v_{2}),\\
\dot{v}_{2} & =-cv_{2}-k_{0}x_{2}-k_{0}(x_{2}-x_{1})-c(v_{2}-v_{1}).
\end{array}\label{eq:ShawPierreModel}
\end{equation}

We first calculate SSMs and backbone curves for this system using
a formulation for continuous dynamical systems, as described in Appendix
A. We then emulate an experimental sampling of the vibrations of system
(\ref{eq:ShawPierreModel}) and reconstruct SSMs and backbone curves
from the sampled data using the discrete methodology described in
Sections \ref{sec:Sampled-nonlinear-vibrations}-\ref{subsec:ModelID}. 

System (\ref{eq:ShawPierreModel}) is analytic, hence we have $r=a$
in our notation. The natural frequencies and damping ratios are 
\[
\omega_{1}=\sqrt{k_{0}},\qquad\omega_{2}=\sqrt{3k_{0}},\qquad\zeta_{1}=\frac{c}{2\sqrt{k_{0}}},\qquad\zeta_{2}=\frac{\sqrt{3}c}{2\sqrt{k_{0}}},
\]
yielding the complex eigenvalues 
\[
\lambda_{1,2}=-\frac{c}{2}\pm i\sqrt{k_{0}\left(1-\frac{c^{2}}{4k_{0}}\right)},\quad\lambda_{3,4}=-\frac{3c}{2}\pm i\sqrt{3k_{0}\left(1-\frac{3c^{2}}{4k_{0}}\right)},
\]
where we have assumed that both modes are underdamped, i.e., $c<2\sqrt{k_{0}/3}.$

For the corresponding two-dimensional modal subspaces $E_{1}$ and
$E_{2},$ Remark \ref{rem: case of proportional damping} gives
\[
\sigma(E_{1})=\mathrm{Int}\,\left[\frac{\mathrm{Re}\,\lambda_{3}}{\mathrm{Re}\,\lambda_{1}}\right]=\mathrm{Int}\,\left[\frac{\frac{\sqrt{3}c}{2\sqrt{k_{0}}}\sqrt{3k_{0}}}{\frac{c}{2\sqrt{k_{0}}}\sqrt{k_{0}}}\right]=3,\qquad\sigma(E_{2})=\mathrm{Int}\,\left[\frac{\mathrm{Re}\,\lambda_{1}}{\mathrm{Re}\,\lambda_{3}}\right]=\mathrm{Int}\,\left[\frac{\frac{c}{2\sqrt{k_{0}}}\sqrt{k_{0}}}{\frac{\sqrt{3}c}{2\sqrt{k_{0}}}\sqrt{3k_{0}}}\right]=0.
\]
Therefore, by Theorem \ref{thm: existence SSM-1} of Appendix B, there
exist two-dimensional, analytic SSMs, $W(E_{1})$ and $W(E_{2})$
that are unique among $C^{4}$ and $C^{1}$ invariant manifolds tangent
to $E_{1}$ and $E_{2}$, respectively, at the origin. 

By the analytic calculations detailed in Appendix C, we obtain the
corresponding backbone curve parametrisations
\[
\omega(\rho_{1})=\frac{1}{2}\left(\sqrt{4k_{0}-c^{2}}+\frac{3\kappa}{\sqrt{4k_{0}-c^{2}}}\rho_{1}^{2}\right),\qquad\mathrm{Amp}(\rho_{1})\approx2\rho_{1},
\]
\[
\omega(\rho_{2})=\frac{1}{2}\left(\sqrt{3\left(4k_{0}-3c^{2}\right)}+\frac{\sqrt{3}\kappa}{\sqrt{4k_{0}-3c^{2}}}\rho_{2}^{2}\right),\qquad\mathrm{Amp}(\rho_{2})\approx2\rho_{2}.
\]

To determine these backbone curves for the damping and stiffness values
$c=0.003$, $k_{0}=1$, and $\kappa=0.5$, we emulate a hammer experiment
that gives an initial condition in the modal subspaces $E_{1}$ and
$E_{2}$ to the full nonlinear system. The precise initial conditions
of the two decaying signals are
\begin{equation}
\boldsymbol{x}^{(1)}(0)=\frac{1}{\sqrt{3}}\left(2,2,0,0\right)^{T}\in E_{1},\hfill\boldsymbol{x}^{(2)}(0)=\frac{1}{3}\left(-2,2,0,0\right)^{T}\in E_{2}.\label{eq:IC for SP}
\end{equation}
 We sample the solutions starting from these points $8000$ times
with the sampling interval $T=0.8.$ In terms of our notation, we
therefore have $P=2$, $M_{1}=M_{2}=8000.$ As observable, we choose
the velocity of the first mass was i.e., let $\varphi(\boldsymbol{x})=v_{1}$,
to emulate an experimental procedure that renders only velocities
(as in our second example below). As minimal embedding dimension for
the sampling map $\boldsymbol{\tilde{F}}(\boldsymbol{\xi})$, Step
1 of the algorithm in Section \ref{sec:Summary-of-algortihm} gives
$2\nu=6$. In the present example, however, we know that $E_{1}$
and $E_{2}$ are properly embedded already in the four-dimensional
system (\ref{eq:ShawPierreModel}), and hence we select $2\nu=4$
instead. 

The red curve in Figure \ref{fig:ShBackb} shows a closed-form quintic
computation (cf. Remark \ref{rem:higher order}) of the backbone curves
$\mathcal{B}_{1}$ and $\mathcal{B}_{2}$ from the data-assimilating
discrete algorithm described in Section \ref{sec:Summary-of-algortihm}.
The two trajectories used as inputs for this algorithm were launched
from the initial conditions (\ref{eq:IC for SP}).

For comparison, the green dashed line in the same figure shows a cubic
analytic computation of the backbone curves based on the continuous-time
(vector-field) formulation we have given in Theorem \ref{thm:coefficients-1}
of Appendix B. Finally, we have used numerical continuation \citep{KnutURL}
at various amplitudes of forcing to find periodic orbits for low damping
with $c=0.0005$. The resulting periodic response amplitudes are shown
in Figure \ref{fig:ShBackb} in blue as functions of the forcing frequency.
The $\mathcal{O}(5)$ backbone curve fits remarkably well with the
peaks of the blue curves, especially considering that these backbone
curves were computed from just two sampled trajectories. The robustness
of the backbone curves is also noteworthy, given that the blue curves
were obtained for substantially lower damping values.

\begin{figure}[th]
\begin{centering}
\includegraphics[width=1\linewidth]{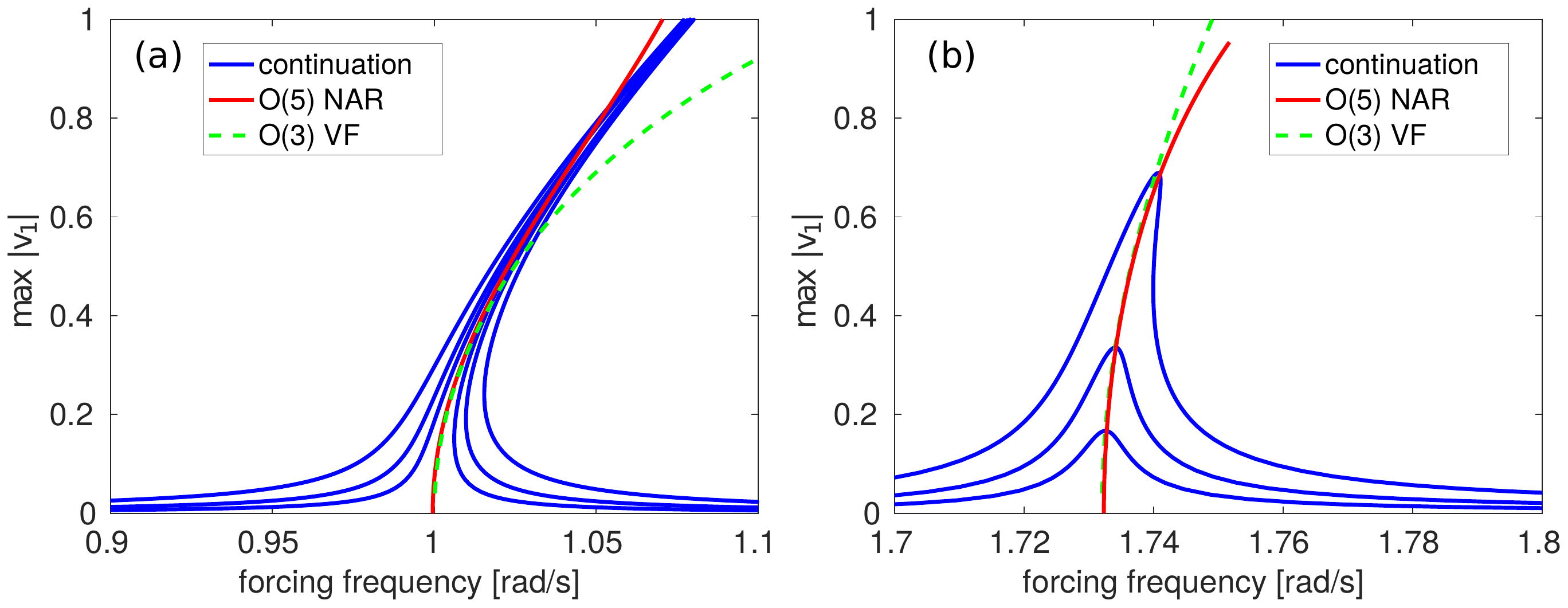}
\par\end{centering}
\caption{Backbone curves and forced response of the mechanical system (\ref{eq:ShawPierreModel})
Blue curves show forced responses of the lightly damped system $c=0.0005$.
Red continuous lines show the fifth-order backbone curves recovered
from our algorithm by sampling two freely decaying trajectories with
initial conditions (\ref{eq:IC for SP}). Green dashed lines show
the $\mathcal{O}(3)$ analytic calculation of the same backbone curves
using Theorem \ref{thm:coefficients-1} of Appendix B.\label{fig:ShBackb}}
\end{figure}

\subsection{Clamped-clamped beam}

We now test the trajectory-data-assimilating backbone-curve reconstruction
algorithm of Section \ref{sec:Summary-of-algortihm} on experimental
data obtained from the vibration tests described in \citep{Ehrhardt2016612}.
We show the experimental device, a beam clamped at both ends, in Figure
\ref{fig:CCBeam}.

\begin{figure}[th]
\begin{centering}
\includegraphics[width=0.8\linewidth]{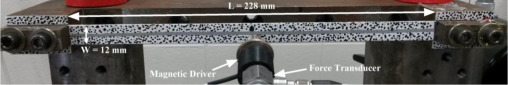}
\par\end{centering}
\caption{The experimental set-up for constructing backbone curves for a clamped-clamped
beam. Reproduced from \citep{Ehrhardt2016612}.\label{fig:CCBeam}}
\end{figure}

The data comprises freely decaying velocity signals measured at the
midpoint of the beam with initial conditions selected near three assumed
SSMs. These initial conditions were obtained experimentally by force
appropriation (cf. the Introduction). The decaying signals were initialised
at maximal response amplitudes obtained from single-frequency force
appropriation. Three signals were assimilated, corresponding to each
natural frequency, which gives $P=3$ in our notation. Each signal
was re-sampled with time period $T=0.97656$ ms. The length of the
three signals were $M_{1}=3892$, $M_{2}=2458$ and $M_{3}=1055$
samples.

The second mode was not analysed in \citep{Ehrhardt2016612}, because
the node of this mode is precisely at the midpoint of the beam, which
can significantly deteriorate measurement accuracy. We list the natural
frequencies identified from the NAR model in Table \ref{tab:CCBnatFreq}.
In the last row of the same table, we also show the spectral quotients
obtained from formula (\ref{eq:sigma from parameters}) for the three
modes.

\begin{table}[th]
\begin{centering}
\begin{tabular}{|c|c|c|c|}
\hline 
Mode & $l=1$ & $l=2$ & $l=3$\tabularnewline
\hline 
\hline 
$\omega_{l}$ {[}Hz{]} & $47.4921$ & $167.1512$ & $368.4577$\tabularnewline
\hline 
$\zeta_{l}$ & $0.1833$ & $0.0183$ & $0.0019$\tabularnewline
\hline 
$\sigma(\mathcal{E}_{l})$ & $0$ & $2$ & $12$\tabularnewline
\hline 
\end{tabular}
\par\end{centering}
\caption{Natural frequencies and damping ratios for the first three modes of
the clamped-clamped beam as determined by our algorithm. Also shown
are the spectral quotients $\sigma(\mathcal{E}_{l})$. The $\omega_{l}$
values are close to those linearly identified in \citep{Ehrhardt2016612},
but the $\zeta_{l}$ values are markedly different. \label{tab:CCBnatFreq}}
\end{table}

Based on Table 1, Theorem \ref{thm: existence SSM} gives a unique
SSM $W(\mathcal{E}_{1})$ within the class of $C^{1}$ manifolds.
This is because the first mode represents the fastest decaying linear
subspace of oscillations, admitting a unique nonlinear continuation
in the form of the fast SSM $W(\mathcal{E}_{1})$. The second (slow)
SSM $W(\mathcal{E}_{2})$ and the third (intermediate) SSM, $W(\mathcal{E}_{3}),$
are only unique among $C^{3}$ and $C^{13}$ invariant manifolds tangent
to the spectral subspaces $\mathcal{E}_{2}$ and $\mathcal{E}_{3}$,
respectively. This suggests that backbone reconstruction techniques
that do not consider the smoothness of the underlying SSM are expected
to show greater uncertainty for the second and the third mode.

We seek to obtain an NAR model for the delay embedding of all three
modes in Table 1. This means we have $\nu=3$, and hence the required
minimal dimension of the reconstructed nonlinear sampling map $\boldsymbol{\tilde{F}}(\boldsymbol{\xi})$
is $2\nu=6.$ We employ a third-order polynomial model $(r=3)$ in
the NAR model of Section \ref{subsec:ModelID}. Accordingly, we construct
the dynamics on the three SSMs up to cubic order (cf. formula (\ref{eq:1DOFred})),
with the Taylor coefficients of $\boldsymbol{W}$ and $\boldsymbol{R}$
computed from the formulas given in Theorem \ref{thm:coefficients}.

Figure \ref{fig:CCBresult} shows the results of our computations.
To be consistent with Ehrhardt and Allen \citep{Ehrhardt2016612},
we compute the response amplitudes by dividing the available instantaneous
velocity amplitudes with their corresponding instantaneous frequencies.
This simple devision, therefore, represents the function $\mathcal{P}$
from the observable space to the relevant coordinate space (cf. Remark
\ref{rem: observed backbone}). The resulting backbone curve of the
first SSM matches well previous results. This is expected, because
this SSM is the most robust among the three SSMs considered here (unique
already among $C^{1}$ invariant manifolds tangent to the spectral
subspace $\mathcal{E}_{1}$). The kink at about $90$ Hz appears to
be an artefact of $\mathcal{O}(3)$ model fitting. Higher amplitude
results for this SSM (not shown) are even less reliable because of
the relative sparsity of the data there.

There is no comparison available from Ehrhardt and Allen \citep{Ehrhardt2016612}
for the second backbone curve, but the backbone curve we compute for
this case is consistent with the instantaneous amplitude-frequency
data (green) inferred from decaying vibrations. 

For the third SSM, there is a noticeable offset between the force
appropriation result and the rest of the curves. Our calculations,
however, match closely the instantaneous amplitude-frequency data,
with the backbone curve obtained from resonance decay. Capturing the
SSM corresponding to this mode uniquely would theoretically require
a high-order, $\mathcal{O}(13)$ approximation. This, however, would
be unfeasible given the limited amount of data available.

\begin{figure}[th]
\begin{centering}
\includegraphics[width=1\linewidth]{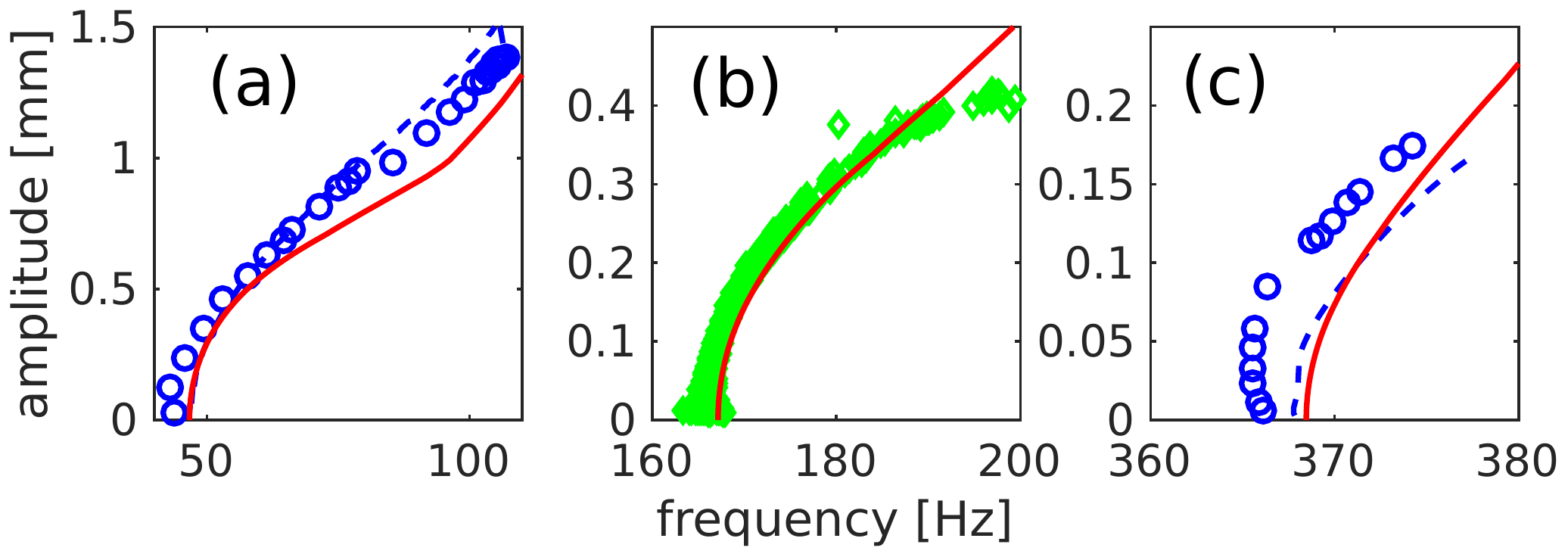}
\par\end{centering}
\caption{The first three backbone curves of a clamped-clamped beam. Red curves:
backbone curves computed from a data-assimilating cubic-order SSM
reduction, as summarised in Section \ref{sec:Summary-of-algortihm};
Blue dashed lines: backbone curves obtained from individual decaying
signals using a Hilbert transform approach \citep{Feldman1997475}.
Blue circles: force-appropriation results using stepped sine forcing.
Green diamonds: Instantaneous amplitude-frequency curves inferred
from decaying vibration data by calculating zero crossings of the
signal to estimate vibration period. Apart from the red curves, all
data was obtained directly from the experiments of Ehrhardt and Allen
\citep{Ehrhardt2016612}. \label{fig:CCBresult}}
\end{figure}

\section{Discussion}

We have developed a method to extract two-dimensional spectral submanifolds
(SSMs) and their associated backbone curves for multi-degree-of-freedom
nonlinear mechanical vibrations. We computed the SSMs explicitly as
two-dimensional invariant manifolds of a low-order, discrete model
system fitted to sampled trajectory data. Restricted to the SSMs,
this model is guaranteed to be conjugate to the full mechanical system
by the classic Takens embedding theorem, as long as the data assimilated
into the model is from a generic observable.

We have illustrated the power of this approach by calculating backbone
curves of the reconstructed dynamics on the SSMs in two examples.
In our first example, a two-degree-of-freedom analytic model, we verified
the trajectory-data-based backbone-curve computation via an analytic
calculation of the same curve for the full, continuous-time system,
as well as by numerical continuation. In our second example, we compared
the data-assimilated construction of the backbone curves with various
experimentally inferred curves and found close agreement. 

To obtain SSMs and their reduced dynamics analytically, we use the
parameterisation method of \citep{CabreLlave2003}, which is generally
not limited to a small neighbourhood of a fixed point. In addition,
the parametrisation method allows for the presence of resonances or
near-resonances that unavoidably arise in underdamped oscillations
(cf. eq. (\ref{eq:near resonance})). This is in contrast with parametrised
SSM constructions based on Sternberg's analytic linearisation theorem
(cf. Cirillo et al. \textbf{\citep{Cirillo2015}}) that exclude any
resonance in the linearised spectrum. When applied in the near-resonant
case, the domain of validity of the analytic linearisation and the
manifolds construction is, therefore, exceedingly small. In addition,
reliance on analytic linearisation excludes the possibility of extracting
backbone curves, which arise from the nonlinear dynamics on the reconstructed
SSM.

The parametrisation method enables us to identify SSMs with high accuracy
on larger domains, even from relatively low-amplitude trajectory samples,
as long as we use high-enough order in the approximations for the
SSMs and its reduced dynamics. This high-enough order ensures the
accurate interrogation of nonlinearities even from low-amplitude signals.
In our examples, a fifth-order computation yielded remarkably accurate
results even for higher-amplitude ranges of the backbone curve, while
a third-order computation was effective for lower-amplitude backbone-curve
ranges.

Our algorithm is devised in a way so that an arbitrary number of decaying
vibrations can be assimilated into the underlying reduced-order discrete
NAR model. Unlike normal forms derived specifically for given modes
of interest, our model incorporates the dynamics of all modes of interest
simultaneously. This should make the reconstructed sampling map $\tilde{\mathcal{\boldsymbol{F}}}$
an ideal tool for use in model-based control. 

We also envisage a closed loop identification of SSMs and backbone
curves, similar to control-based continuation techniques \citep{SieberPRL2008}.
In this case, a measure of invariance derived from equation (\ref{eq:invariancecond})
would serve as a test functional. 

\section*{Acknowledgements}

The authors thank David Barton, Alan Champneys, Gaetan Kerschen, Simon
Neild and Alex Vakakis for very helpful discussions. We are also grateful
to Thomas Breunung for catching several typos and an error in an earlier
draft of the manuscript. The work of R.S. and D.E. was partly funded
by EPSRC under the Engineering Nonlinearity Programme grant no.~EP/K003836/1.

\section{Appendix A: Proof of Theorem \ref{thm:coefficients}}

By the relationship (\ref{eq:invariancecond0}), we need to solve
the algebraic equation 
\begin{equation}
\boldsymbol{\Lambda}\boldsymbol{W}+\boldsymbol{G}\circ\boldsymbol{W}=\boldsymbol{W}\circ\boldsymbol{R}\label{eq:invariancecond}
\end{equation}
for the unknown Taylor coefficients $w_{j}^{\boldsymbol{s}}$ and
$r_{j}^{\boldsymbol{s}}$. We carry this out step by step for increasing
powers of $\boldsymbol{z}$:
\begin{description}
\item [{$\mathcal{O}\left(\left|\boldsymbol{z}\right|\right)$:}] Since
the Taylor series of $\boldsymbol{G}$ starts with second-order terms,
the first-order monomials of $\boldsymbol{z}$ arising from substitution
into (\ref{eq:invariancecond}) satisfy $\boldsymbol{\Lambda}\boldsymbol{W}=\boldsymbol{W}\circ\boldsymbol{R,}$
which simplifies to $\boldsymbol{\Lambda}\boldsymbol{W}=\boldsymbol{W}\boldsymbol{\Lambda}$,
because the linear part of $\boldsymbol{R}$ is equal to the $\mathrm{diag}\left\{ \mu_{\ell},\bar{\mu}_{\ell}\right\} $
block of the diagonal matrix $\boldsymbol{\Lambda}$ (cf. (\ref{eq:1DOFred})).
Consequently, (\ref{eq:invariancecond}) can be written at leading
order as 
\begin{equation}
\mu_{j}w_{j}^{(1,0)}=w_{j}^{(1,0)}\mu_{\ell},\qquad\bar{\mu}_{j}w_{j+1}^{(0,1)}=w_{j+1}^{(0,1)}\bar{\mu}_{\ell},\label{eq:first order W}
\end{equation}
 whose simplest solution is 
\begin{equation}
w_{j}^{(1,0)}=\delta_{j\ell},\qquad w_{j}^{(0,1)}=\delta_{j(\ell+1)},\label{eq:first order W nonresonant choice}
\end{equation}
with $\delta_{j\ell}$ denoting the Kronecker delta. This proves the
formulas for the first-order coefficients in Theorem \ref{thm:coefficients}.
We note that $w_{\ell}^{(1,0)}$ and $w_{\ell+1}^{(0,1)}$ are only
determined up to a constant, which we have chosen to be equal to $1$. 
\item [{$\mathcal{O}\left(\left|\boldsymbol{z}\right|^{2}\right)$:}] Since
$r_{j}^{\boldsymbol{s}}=0$ for $\left|\boldsymbol{s}\right|=2$ by
(\ref{eq:1DOFred}), the quadratic terms in $\boldsymbol{\left|z\right|}$
on the right-hand side of (\ref{eq:invariancecond}) only arise from
the substitution of linear terms of $\boldsymbol{R}(\boldsymbol{z})$
into the quadratic terms of $\boldsymbol{W}(\boldsymbol{z})$. As
a consequence, equating the coefficients of $\mathcal{O}\boldsymbol{\left(\left|z\right|^{2}\right)}$
terms on both sides of (\ref{eq:invariancecond}) gives the equation
\begin{equation}
\mu_{j}w_{j}^{\left(s_{1},s_{2}\right)}+g_{j}^{\left(s_{1}@\ell,s_{2}@\left(\ell+1\right)\right)}=\mu_{\ell}^{s_{1}}\bar{\mu}_{\ell}^{s_{2}}w_{j}^{\left(s_{1},s_{2}\right)},\qquad\left|\boldsymbol{s}\right|=2,\quad\label{eq:DISC2ndOrd}
\end{equation}
whose solution for $w_{j}^{\left(s_{1},s_{2}\right)}$ is
\begin{equation}
w_{j}^{(s_{1},s_{2})}=\frac{g_{j}^{\left(s_{1}@\ell,s_{2}@\left(\ell+1\right)\right)}}{\mu_{\ell}^{s_{1}}\bar{\mu}_{\ell}^{s_{2}}-\mu_{j}},\quad\qquad\left|\boldsymbol{s}\right|=2,\label{eq:W_j second order general}
\end{equation}
proving the formulas for the second-order coefficients $w_{j}^{(2,0)},$
$w_{j}^{(1,1)}$ and $w_{j}^{(0,2)}$ in the statement of Theorem
\ref{thm:coefficients}. Note that the denominator in (\ref{eq:W_j second order general})
is guaranteed to be nonzero by the nonresonance condition (\ref{eq:nonresonance}).
\item [{$\mathcal{O}\left(\left|\boldsymbol{z}\right|^{3}\right)$:}] We
write out the $j^{th}$ coordinates in the three terms of eq. (\ref{eq:invariancecond})
in detail to obtain the following cubic terms:
\begin{eqnarray}
\left(\boldsymbol{\Lambda}\boldsymbol{W}\right)_{j}^{(3)} & = & \mu_{j}\left(w_{j}^{(3,0)}z_{\ell}^{3}+w_{j}^{(2,1)}z_{\ell}^{2}\bar{z}_{\ell}+w_{j}^{(1,2)}z_{\ell}\bar{z}_{\ell}^{2}+w_{j}^{(0,3)}\bar{z}_{\ell}^{3}\right),\label{eq:coeff1}\\
\left(\boldsymbol{G}\circ\boldsymbol{W}\right)_{j}^{(3)} & = & \sum_{\begin{array}{c}
\begin{smallmatrix}\left|\boldsymbol{m}\right|=2\end{smallmatrix}\\
\begin{smallmatrix}y^{\boldsymbol{m}}=\mathcal{O}\left(\boldsymbol{\left|z\right|^{3}}\right)\end{smallmatrix}
\end{array}}g_{j}^{\boldsymbol{m}}\boldsymbol{y}^{\boldsymbol{m}}+\sum_{\begin{array}{c}
\begin{smallmatrix}\left|\boldsymbol{m}\right|=3\end{smallmatrix}\\
\begin{smallmatrix}y^{\boldsymbol{m}}=\mathcal{O}\left(\boldsymbol{\left|z\right|^{3}}\right)\end{smallmatrix}
\end{array}}g_{j}^{\boldsymbol{m}}\boldsymbol{y}^{\boldsymbol{m}},\label{eq:coeff2}\\
\left(\boldsymbol{W}\circ\boldsymbol{R}\right)_{j}^{(3)} & = & \sum_{\begin{array}{c}
\begin{smallmatrix}\left|\boldsymbol{s}\right|=1,2,3\end{smallmatrix}\\
\begin{smallmatrix}\boldsymbol{\,r}^{\boldsymbol{s}}=\mathcal{O}\left(\boldsymbol{\left|z\right|^{3}}\right)\end{smallmatrix}
\end{array}}w_{j}^{(s_{1},s_{2})}r_{\ell}^{s_{1}}\bar{r}_{\ell}^{s_{2}}\nonumber \\
 & = & \sum_{\begin{array}{c}
\begin{smallmatrix}\left|\boldsymbol{s}\right|=1,2,3\end{smallmatrix}\\
\begin{smallmatrix}\boldsymbol{\,r}^{\boldsymbol{s}}=\mathcal{O}\left(\boldsymbol{\left|z\right|^{3}}\right)\end{smallmatrix}
\end{array}}w_{j}^{(s_{1},s_{2})}\left(\mu_{\ell}z_{\ell}+\beta_{\ell}z_{\ell}^{2}\bar{z}_{\ell}\right)^{s_{1}}\left(\bar{\mu}_{\ell}\bar{z}_{\ell}+\bar{\beta_{\ell}}z_{\ell}\bar{z}_{\ell}^{2}\right)^{s_{2}}\nonumber \\
 & = & w_{j}^{(3,0)}\mu_{\ell}^{3}z_{\ell}^{3}+w_{j}^{(0,3)}\bar{\mu}_{\ell}^{3}\bar{z}_{\ell}^{3}+w_{j}^{(2,1)}\mu_{\ell}^{2}\bar{\mu}_{\ell}z_{\ell}^{2}\bar{z}_{\ell}\nonumber \\
 &  & +w_{j}^{(1,2)}\mu_{\ell}\bar{\mu}_{\ell}^{2}z_{\ell}\bar{z}_{\ell}^{2}+w_{j}^{(1,0)}\beta_{\ell}z_{\ell}^{2}\bar{z}_{\ell}+w_{j}^{(0,1)}\bar{\beta_{\ell}}z_{\ell}\bar{z}_{\ell}^{2}.\label{eq:coeff3}
\end{eqnarray}
We now write out the individual terms in eq. (\ref{eq:coeff2}). For
$\boldsymbol{y}=\boldsymbol{W}(\boldsymbol{z})$, we have
\begin{eqnarray*}
y_{p}y_{q} & = & \left(\delta_{p\ell}z_{\ell}+\delta_{p\left(\ell+1\right)}\bar{z}_{\ell}+w_{p}^{(2,0)}z_{\ell}^{2}+w_{p}^{(1,1)}z_{\ell}\bar{z}_{\ell}+w_{p}^{(0,2)}\bar{z}_{\ell}^{2}\right)\\
 &  & \times\left(\delta_{q\ell}z_{\ell}+\delta_{q\left(\ell+1\right)}\bar{z}_{\ell}+w_{q}^{(2,0)}z_{\ell}^{2}+w_{q}^{(1,1)}z_{\ell}\bar{z}_{\ell}+w_{q}^{(0,2)}\bar{z}_{\ell}^{2}\right)\\
 & = & \left(\delta_{p\ell}w_{q}^{(2,0)}+\delta_{q\ell}w_{p}^{(2,0)}\right)z_{\ell}^{3}+\left(\delta_{p\ell}w_{q}^{(1,1)}+\delta_{q\ell}w_{p}^{(1,1)}+\delta_{p\left(\ell+1\right)}w_{q}^{(2,0)}+\delta_{q\left(\ell+1\right)}w_{p}^{(2,0)}\right)z_{\ell}^{2}\bar{z}_{\ell}\\
 &  & +\left(\delta_{p\ell}w_{q}^{(0,2)}+\delta_{q\ell}w_{p}^{(0,2)}+\delta_{p\left(\ell+1\right)}w_{q}^{(1,1)}+\delta_{q\left(\ell+1\right)}w_{p}^{(1,1)}\right)z_{\ell}\bar{z}_{\ell}^{2}+\left(\delta_{p\left(\ell+1\right)}w_{q}^{(0,2)}+\delta_{q\left(\ell+1\right)}w_{p}^{(0,2)}\right)\bar{z}_{\ell}^{3}\\
 &  & +\mathcal{O}\left(\left|\boldsymbol{z}\right|^{2},\left|\boldsymbol{z}\right|^{4}\right),
\end{eqnarray*}
thus, for the first sum in (\ref{eq:coeff2}), we obtain 
\begin{eqnarray}
\sum_{\begin{array}{c}
\begin{smallmatrix}\left|\boldsymbol{m}\right|=2\end{smallmatrix}\\
\begin{smallmatrix}y^{\boldsymbol{m}}=\mathcal{O}\left(\boldsymbol{\left|z\right|^{3}}\right)\end{smallmatrix}
\end{array}}g_{j}^{\boldsymbol{m}}\boldsymbol{y}^{\boldsymbol{m}} & = & \sum_{p,q=1}^{2\nu}\frac{g_{j}^{(1@p,1@q)}}{2-\delta_{pq}}\left[\left(\delta_{p\ell}w_{q}^{(2,0)}+\delta_{q\ell}w_{p}^{(2,0)}\right)z_{\ell}^{3}+\left(\delta_{p\left(\ell+1\right)}w_{q}^{(0,2)}+\delta_{q\left(\ell+1\right)}w_{p}^{(0,2)}\right)\bar{z}_{\ell}^{3}\right]\nonumber \\
 &  & +\sum_{p,q=1}^{2\nu}\frac{g_{j}^{(1@p,1@q)}}{2-\delta_{pq}}\left[\delta_{p\ell}w_{q}^{(1,1)}+\delta_{q\ell}w_{p}^{(1,1)}+\delta_{p\left(\ell+1\right)}w_{q}^{(2,0)}+\delta_{q\left(\ell+1\right)}w_{p}^{(2,0)}\right]z_{\ell}^{2}\bar{z}_{\ell}\nonumber \\
 &  & +\sum_{p,q=1}^{2\nu}\frac{g_{j}^{(1@p,1@q)}}{2-\delta_{pq}}\left[\delta_{p\ell}w_{q}^{(0,2)}+\delta_{q\ell}w_{p}^{(0,2)}+\delta_{p\left(\ell+1\right)}w_{q}^{(1,1)}+\delta_{q\left(\ell+1\right)}w_{p}^{(1,1)}\right]z_{\ell}\bar{z}_{\ell}^{2}\nonumber \\
 & =2 & \sum_{q=1}^{2\nu}\frac{g_{j}^{(1@\ell,1@q)}}{2-\delta_{\ell q}}w_{q}^{(2,0)}z_{\ell}^{3}+2\sum_{q=1}^{2\nu}\frac{g_{j}^{(1@(\ell+1),1@q)}}{2-\delta_{(\ell+1)q}}w_{q}^{(0,2)}\bar{z}_{\ell}^{3}\nonumber \\
 &  & +2\sum_{q=1}^{2\nu}\left[\frac{g_{j}^{(1@\ell,1@q)}}{2-\delta_{\ell q}}w_{q}^{(1,1)}+\frac{g_{j}^{(1@(\ell+1),1@q)}}{2-\delta_{(\ell+1)q}}w_{q}^{(2,0)}\right]z_{\ell}^{2}\bar{z}_{\ell}\nonumber \\
 &  & +2\sum_{q=1}^{2\nu}\left[\frac{g_{j}^{(1@\ell,1@q)}}{2-\delta_{\ell q}}w_{q}^{(0,2)}+\frac{g_{j}^{(1@(\ell+1),1@q)}}{2-\delta_{(\ell+1)q}}w_{q}^{(1,1)}\right]z_{\ell}\bar{z}_{\ell}^{2}.\label{eq:secondorder}
\end{eqnarray}
For the second sum in (\ref{eq:coeff2}), we have
\begin{eqnarray}
\sum_{\begin{array}{c}
\begin{smallmatrix}\left|\boldsymbol{m}\right|=3\end{smallmatrix}\\
\begin{smallmatrix}y^{\boldsymbol{m}}=\mathcal{O}\left(\boldsymbol{\left|z\right|^{3}}\right)\end{smallmatrix}
\end{array}}g_{j}^{\boldsymbol{m}}\boldsymbol{y}^{\boldsymbol{m}} & = & \underbrace{\sum_{\begin{array}{c}
\begin{smallmatrix}p=1\end{smallmatrix}\\
\begin{smallmatrix}y^{\boldsymbol{m}}=\mathcal{O}\left(\boldsymbol{\left|z\right|^{3}}\right)\end{smallmatrix}
\end{array}}^{2\nu}g_{j}^{(3@p)}y_{p}^{3}}_{A}+\underbrace{\sum_{\begin{array}{c}
\begin{smallmatrix}p,q=1\\
p\neq q
\end{smallmatrix}\\
\begin{smallmatrix}y^{\boldsymbol{m}}=\mathcal{O}\left(\boldsymbol{\left|z\right|^{3}}\right)\end{smallmatrix}
\end{array}}^{2\nu}g_{j}^{(2@p,1@q)}y_{p}^{2}y_{q}}_{B}\label{eq:thirdorder}\\
 &  & +\underbrace{\sum_{\begin{array}{c}
\begin{smallmatrix}p,q,u=1\\
p\neq q;\,p,q\neq u
\end{smallmatrix}\\
\begin{smallmatrix}y^{\boldsymbol{m}}=\mathcal{O}\left(\boldsymbol{\left|z\right|^{3}}\right)\end{smallmatrix}
\end{array}}^{2\nu}g_{j}^{(1@p,1@q,1@u)}y_{p}y_{q}y_{u}}_{C}.\nonumber 
\end{eqnarray}
Working out these expressions in detail, we find that 
\begin{eqnarray}
A & = & \sum_{p=1}^{2\nu}g_{j}^{(3@p)}\left(\delta_{p\ell}^{3}z_{\ell}^{3}+3\delta_{p\ell}^{2}\delta_{p\left(\ell+1\right)}z_{\ell}^{2}\bar{z}_{\ell}+3\delta_{p\ell}\delta_{p\left(\ell+1\right)}^{2}z_{\ell}\bar{z}_{\ell}^{2}+\delta_{p\left(\ell+1\right)}^{3}\bar{z}_{\ell}^{3}\right)\nonumber \\
 & = & g_{j}^{(3@\ell)}z_{\ell}^{3}+g_{j}^{(3@\left(\ell+1\right))}\bar{z}_{\ell}^{3},\nonumber \\
B & = & \sum_{\begin{array}{c}
\begin{smallmatrix}p,q=1\\
p\neq q
\end{smallmatrix}\end{array}}^{2\nu}g_{j}^{(2@p,1@q)}\left(\delta_{p\ell}^{2}z_{\ell}^{2}+2\delta_{p\ell}\delta_{p\left(\ell+1\right)}z_{\ell}\bar{z}_{\ell}+\delta_{p\left(\ell+1\right)}^{2}\bar{z}_{\ell}^{2}\right)\left(\delta_{q\ell}z_{\ell}+\delta_{q\left(\ell+1\right)}\bar{z}_{\ell}\right)\nonumber \\
 & = & \sum_{\begin{array}{c}
\begin{smallmatrix}p,q=1\\
p\neq q
\end{smallmatrix}\end{array}}^{2\nu}g_{j}^{(2@p,1@q)}\left(\delta_{p\ell}^{2}\delta_{q\ell}z_{\ell}^{3}+2\delta_{p\ell}\delta_{q\ell}\delta_{p\left(\ell+1\right)}z_{\ell}^{2}\bar{z}_{\ell}+\delta_{q\ell}\delta_{p\left(\ell+1\right)}^{2}z_{\ell}\bar{z}_{\ell}^{2}\right)\nonumber \\
 & + & \sum_{\begin{array}{c}
\begin{smallmatrix}p,q=1\\
p\neq q
\end{smallmatrix}\end{array}}^{2\nu}g_{j}^{(2@p,1@q)}\left(\delta_{p\ell}^{2}\delta_{q\left(\ell+1\right)}z_{\ell}^{2}\bar{z}_{\ell}+2\delta_{p\ell}\delta_{p\left(\ell+1\right)}\delta_{q\left(\ell+1\right)}z_{\ell}\bar{z}_{\ell}^{2}+\delta_{p\left(\ell+1\right)}^{2}\delta_{q\left(\ell+1\right)}\bar{z}_{\ell}^{3}\right)\nonumber \\
 & = & g_{j}^{(2@\ell,1@(\ell+1))}z_{\ell}^{2}\bar{z}_{\ell}+g_{j}^{(2@(\ell+1),1@\ell)}z_{\ell}\bar{z}_{\ell}^{2},\nonumber \\
C & = & \sum_{\begin{array}{c}
\begin{smallmatrix}p,q,u=1\\
p\neq q;\,p,q\neq u
\end{smallmatrix}\end{array}}^{2\nu}g_{j}^{(1@p,1@q,1@u)}\left(\delta_{p\ell}z_{\ell}+\delta_{p\left(\ell+1\right)}\bar{z}_{\ell}\right)\left(\delta_{q\ell}z_{\ell}+\delta_{q\left(\ell+1\right)}\bar{z}_{\ell}\right)\left(\delta_{u\ell}z_{\ell}+\delta_{u\left(\ell+1\right)}\bar{z}_{\ell}\right)\nonumber \\
 & = & 0.\label{eq:ABC}
\end{eqnarray}
Substituting the expressions for $A,$ $B$ and $C$ into (\ref{eq:thirdorder}),
then substituting (\ref{eq:thirdorder}) and (\ref{eq:secondorder})
into the invariance condition (\ref{eq:invariancecond}), we equate
equal powers of $\boldsymbol{z}$ to obtain the following linear equations
for the cubic coefficients of the mapping $\boldsymbol{W}$ and of
the mapping $\boldsymbol{R}$:
\begin{eqnarray}
\mathcal{O}\left(z_{\ell}^{3}\right) & : & \mu_{j}w_{j}^{(3,0)}+2\sum_{q=1}^{2\nu}\frac{g_{j}^{(1@\ell,1@q)}}{2-\delta_{\ell q}}w_{q}^{(2,0)}+g_{j}^{(3@\ell)}\nonumber \\
 &  & =w_{j}^{(3,0)}\mu_{\ell}^{3},\nonumber \\
\mathcal{O}\left(z_{\ell}^{2}\bar{z}_{\ell}\right) & : & \mu_{j}w_{j}^{(2,1)}+2\sum_{q=1}^{2\nu}\left[\frac{g_{j}^{(1@\ell,1@q)}}{2-\delta_{\ell q}}w_{q}^{(1,1)}+\frac{g_{j}^{(1@(\ell+1),1@q)}}{2-\delta_{(\ell+1)q}}w_{q}^{(2,0)}\right]+g_{j}^{(2@\ell,1@(\ell+1))}\nonumber \\
 &  & =w_{j}^{(2,1)}\mu_{\ell}^{2}\bar{\mu}_{\ell}+w_{j}^{(1,0)}\beta_{\ell},\nonumber \\
\mathcal{O}\left(z_{\ell}\bar{z}_{\ell}^{2}\right) & : & \mu_{j}w_{j}^{(1,2)}+2\sum_{q=1}^{2\nu}\frac{g_{j}^{(1@(\ell+1),1@q)}}{2-\delta_{(\ell+1)q}}+g_{j}^{(2@(\ell+1),1@\ell)}\nonumber \\
 &  & =w_{j}^{(1,2)}\mu_{\ell}\bar{\mu}_{\ell}^{2}+w_{j}^{(0,1)}\bar{\beta_{\ell}},\nonumber \\
\mathcal{O}\left(\bar{z}_{\ell}^{3}\right) & : & \mu_{j}w_{j}^{(0,3)}+2\sum_{q=1}^{2\nu}\frac{g_{j}^{(1@(\ell+1),1@q)}}{2-\delta_{(\ell+1)q}}w_{q}^{(0,2)}+g_{j}^{(3@\left(\ell+1\right))}\nonumber \\
 &  & =w_{j}^{(0,3)}\bar{\mu}_{\ell}^{3}.\label{eq:allterms}
\end{eqnarray}
From the first and last equation in (\ref{eq:allterms}), we obtain
\begin{eqnarray}
w_{j}^{(3,0)} & = & \frac{\sum_{q=1}^{2\nu}\left(1+\delta_{\ell q}\right)g_{j}^{(1@\ell,1@q)}w_{q}^{(2,0)}+g_{j}^{(3@\ell)}}{\mu_{\ell}^{3}-\mu_{j}},\nonumber \\
w_{j}^{(0,3)} & = & \frac{\sum_{q=1}^{2\nu}\left(1+\delta_{(\ell+1)q}\right)g_{j}^{(1@(\ell+1),1@q)}w_{q}^{(0,2)}+g_{j}^{(3@\left(\ell+1\right))}}{\bar{\mu}_{\ell}^{3}-\mu_{j}}.\label{eq:w_j30,03}
\end{eqnarray}
We select $j\neq\ell$ and assume that there is no first-order near-resonance
(or exact resonance) involving the eigenvalues $\mu_{\ell}$ and $\mu_{j}$
(stated as $\mu_{j}\not\approx\mu_{\ell}$ under the assumptions of
the theorem). Recalling $w_{j}^{(1,0)}=\delta_{j\ell}$, we then obtain
from the second equation of (\ref{eq:allterms}) that 
\begin{equation}
w_{j}^{(2,1)}=\frac{\sum_{q=1}^{2\nu}\left[\left(1+\delta_{\ell q}\right)g_{j}^{(1@\ell,1@q)}w_{q}^{(1,1)}+\left(1+\delta_{(\ell+1)q}\right)g_{j}^{(1@(\ell+1),1@q)}w_{q}^{(2,0)}\right]+g_{j}^{(2@\ell,1@(\ell+1))}}{\mu_{\ell}^{2}\bar{\mu}_{\ell}-\mu_{j}},\label{eq:W_j21 not l}
\end{equation}
whenever $\quad\delta_{j\ell}=0$. Similarly, selecting $j\neq\ell+1$,
assuming no first-order near-resonance (or exact resonance) involving
the eigenvalues $\mu_{\ell+1}$ and $\mu_{j}$ (i.e., $\mu_{j}\not\approx\mu_{\ell+1}$),
and recalling $w_{j}^{(0,1)}=\delta_{j(\ell+1)}$, we obtain from
the third equation of (\ref{eq:allterms}) that
\begin{eqnarray}
w_{j}^{(1,2)} & = & \frac{\sum_{q=1}^{2\nu}\left[\left(1+\delta_{\ell q}\right)g_{j}^{(1@\ell,1@q)}w_{q}^{(0,2)}+\left(1+\delta_{(\ell+1)q}\right)g_{j}^{(1@(\ell+1),1@q)}w_{q}^{(1,1)}\right]+g_{j}^{(2@(\ell+1),1@\ell)}}{\mu_{\ell}\bar{\mu}_{\ell}^{2}-\mu_{j}},\nonumber \\
\label{eq:w_j12 not l+1}
\end{eqnarray}
whenever $\delta_{j(\ell+1)}=0$. Next we select $j=\ell$ in the
second equation of (\ref{eq:allterms}), and select $j=\ell+1$ in
the third equation of (\ref{eq:allterms}). These choices force us
to select 
\begin{equation}
w_{\ell}^{(2,1)}=0,\quad\delta_{j\ell}=1,\qquad w_{\ell+1}^{(1,2)}=0,\quad\delta_{j(\ell+1)}=1,\label{eq:w_j cubic vanishing}
\end{equation}
in these equations to avoid small denominators arising from the near-resonances.
Then the second equation of (\ref{eq:allterms}) with $j=\ell$ gives
the solution 
\begin{equation}
\beta_{\ell}=\sum_{q=1}^{2\nu}\left[\left(1+\delta_{\ell q}\right)g_{\ell}^{(1@\ell,1@q)}w_{q}^{(1,1)}+\left(1+\delta_{(\ell+1)q}\right)g_{\ell}^{(1@(\ell+1),1@q)}w_{q}^{(2,0)}\right]+g_{\ell}^{(2@\ell,1@(\ell+1))}.\label{eq:beta proof final}
\end{equation}
But equations (\ref{eq:w_j30,03})-(\ref{eq:beta proof final}) prove
the formulas for the cubic coefficients of $W_{j}$ and $\beta_{\ell}$
in the statement of Theorem \ref{thm:coefficients}. 
\end{description}

\section{Appendix B: Analogous results for continuous dynamical systems}

Here we discuss spectral submanifolds, backbone curves and their leading-order
computation for continuous dynamical systems. The formulas we derive
are useful for benchmarking our data-based SSM and backbone-curve
approach on exactly known mechanical models. The concepts and formulas
derived here, however, are also of independent interest in computing
the dynamics on SSMs in analytically defined mechanical models.

We start with the continuous analogue of the the complex mapping(\ref{eq:DISCdiag}),
which is a complex differential equation of the form

\begin{eqnarray}
\dot{\boldsymbol{y}} & = & \boldsymbol{\Lambda}\boldsymbol{y}+\boldsymbol{G}(\boldsymbol{y}),\qquad\boldsymbol{y}\in\mathbb{C}^{2\nu},\qquad\boldsymbol{\Lambda}=\mathrm{diag}(\lambda_{1},\ldots,\lambda_{2\nu}),\qquad\lambda_{2l-1}=\bar{\lambda}_{2l},\quad l=1,\ldots,\nu,\label{eq:DISCdiag-1}\\
 &  & \boldsymbol{G}(\boldsymbol{y})=\mathcal{O}\left(\left|\boldsymbol{y}\right|^{2}\right),\nonumber 
\end{eqnarray}
with a fixed point at $\boldsymbol{y}=\boldsymbol{0}$, and with a
class $C^{r}$ function $\boldsymbol{G}$. The eigenvalues of $\boldsymbol{\Lambda}$
are ordered so that 
\begin{equation}
\mathrm{Re}\lambda_{2\nu}\leq\ldots\leq\mathrm{Re}\lambda_{1}<0,\label{eq:assumption on eigenvalues of ODE-1}
\end{equation}
and hence $\boldsymbol{y}=\boldsymbol{0}$ is asymptotically stable.
If (\ref{eq:DISCdiag-1}) is the equivalent first-order complexified
form of a mechanical system of the form (\ref{eq:starting point}),
then we specifically have $\nu=n$. If, furthermore, the mechanical
system has linear and weak proportional damping, then we can write
\begin{equation}
\mathrm{Re}\lambda_{l}=-\zeta_{l}\omega_{l},\qquad\mathrm{Im}\lambda_{l}=\sqrt{1-\zeta_{l}^{2}}\omega_{l},\label{eq:imre vector field}
\end{equation}
with $\zeta_{l}$ and $\omega_{l}$ denoting Lehr's damping ratio
and undamped natural frequency, respectively, for the $l^{th}$ mode
of the linearised system at the $\boldsymbol{q}=\boldsymbol{0}$ equilibrium.

Finally, we assume that $\mathcal{E}$ is a two-dimensional spectral
subspace (eigenspace) of the operator $\boldsymbol{\Lambda}$, corresponding
to the complex pair of simple eigenvalues $\lambda_{\ell}=\overline{\lambda}_{\ell+1}$
for some $\ell\in[1,2\nu-1]$. 

\subsection{Existence and uniqueness of SSMs}

Following \citep{Haller2016}, we address this issue via the following
definition:

\begin{defn}
A \emph{spectral submanifold} (SSM) $W(\mathcal{E})$ corresponding
to a spectral subspace $\mathcal{E}$ of $\boldsymbol{\Lambda}$ is
\end{defn}
\begin{description}
\item [{(i)}] an invariant manifold of the dynamical system (\ref{eq:DISCdiag-1})
that is tangent to $\mathcal{E}$ at $\boldsymbol{y}=\boldsymbol{0}$
and has the same dimension as $\mathcal{E}$;
\item [{(ii)}] strictly smoother than any other invariant manifold of (\ref{eq:DISCdiag-1})
satisfying (i).
\end{description}

We now recall from Haller and Ponsioen \citep{Haller2016} the specific
existence and uniqueness result pertaining to two-dimensional SSMs,
deducible from the more general results of Cabré et al. \citep{CabreLlave2003}.
The \emph{relative spectral quotient }of\emph{ $\mathcal{E}$ }is
now defined as the positive integer\foreignlanguage{english}{ }
\begin{eqnarray}
\sigma(\mathcal{E}) & = & \mathrm{Int}\,\left[\frac{\underset{l\neq\ell,\ell+1}{\min}\mathrm{Re}\lambda_{l}}{\mathrm{Re}\lambda_{\ell}}\right]\in\mathbb{N}^{+},\label{eq:relative_sigma-1}
\end{eqnarray}
whose meaning is the same as pointed out after formula (\ref{eq:relative_sigma-1})
for mappings. In case of a proportionally damped mechanical system,
one may use the formulas (\ref{eq:imre vector field}) and conclude
that Remark \ref{rem: case of proportional damping} continues to
provide the correct specific form of $\sigma(\mathcal{E})$ in this
case.

We again assume that
\begin{equation}
\sigma(\mathcal{E})\leq r,\label{eq:spectral condition-1}
\end{equation}
and that no resonance relationships up to order $\sigma(\mathcal{E})$
hold between the eigenvalues $\lambda_{\ell},\overline{\lambda}_{\ell+1}$
and the rest of the spectrum of $\boldsymbol{\Lambda}$, i.e., 

\begin{equation}
s_{1}\lambda_{\ell}+s_{2}\bar{\lambda}_{\ell}\neq\lambda_{j},\qquad\forall j\neq\ell,\ell+1,\qquad2\le s_{1}+s_{2}\le\sigma\left(\mathcal{E}\right).\label{eq:nonresonance-1}
\end{equation}
The alternative form of this nonresonance condition given in Remark
\ref{rem: case of proportional damping} again applies whenever formulas
(\ref{eq:relative_sigma-1}) hold. 
\begin{thm}
\label{thm: existence SSM-1}Assume that conditions (\ref{eq:spectral condition-1})-(\ref{eq:nonresonance-1})
are satisfied. Then following statements hold:
\end{thm}
\begin{description}
\item [{(i)}] There exists an SSM, $W\left(\mathcal{E}\right),$ for the
dynamical system (\ref{eq:DISCdiag-1}) that is tangent to the invariant
subspace $\mathcal{E}$ at the $\boldsymbol{y}=\boldsymbol{0}$ fixed
point. 
\item [{(ii)}] The invariant manifold $W\left(\mathcal{E}\right)$ is class
$C^{r}$ smooth and unique among all two-dimensional, class $C^{\sigma\left(\mathcal{E}\right)+1}$
invariant manifolds of (\ref{eq:DISCdiag-1}) that are tangent to
$\mathcal{E}$ at $\boldsymbol{y}=\boldsymbol{0}$. 
\item [{(iii)}] The SSM $W\left(\mathcal{E}\right)$ can be viewed as a
$C^{r}$ immersion of an open set $\mathcal{U}\subset\mathbb{C}^{2}$
into the phase space $\mathbb{C}^{2\nu}$ of system (\ref{eq:DISCdiag-1})
via a map 
\begin{eqnarray}
\boldsymbol{W}:\mathcal{U}\subset\mathbb{C}^{2} & \to & \mathbb{C}^{2\nu},\qquad\boldsymbol{W}\left(\mathcal{U}\right)=W\left(\mathcal{E}\right).\label{eq:W embedding-1}
\end{eqnarray}
\item [{(iv)}] There exists $C^{r}$ polynomial function $\boldsymbol{R}\colon\mathcal{U}\to\mathcal{U}$
such that 
\begin{equation}
\boldsymbol{\Lambda}\boldsymbol{W}+\boldsymbol{G}\circ\boldsymbol{W}=D\boldsymbol{W}\boldsymbol{R},\label{eq:invariancecond0-1}
\end{equation}
 i.e., the dynamics on the SSM, expressed in coordinates $\boldsymbol{z}=(z_{\ell},\bar{z}_{\ell})\in\mathcal{U},$
is governed by the polynomial ODE
\[
\dot{\boldsymbol{z}}=\boldsymbol{R}(\boldsymbol{z}),\qquad D\boldsymbol{R}(0)=\mathrm{diag}(\lambda_{\ell},\bar{\lambda}_{\ell}),
\]
whose right-hand side has only terms up to order $\mathcal{O}\left(\left|\boldsymbol{z}\right|^{\sigma\left(\mathcal{E}\right)}\right)$.
\item [{(v)}] Under the further internal non-resonance assumption 
\begin{equation}
s_{1}\lambda_{\ell}+s_{2}\bar{\lambda}_{\ell}\neq\lambda_{j},\qquad j=\ell,\ell+1,\qquad2\le s_{1}+s_{2}\le\sigma\left(\mathcal{E_{\ell}}\right),\label{eq:strict nonresonance-1}
\end{equation}
\emph{within} $\mathcal{E}$, the mapping $\boldsymbol{W}$ in \ref{eq:W embedding-1}
can be selected such that the $j^{th}$ coordinate component $R_{j}$
of $\boldsymbol{R}$ does not contain the term $(z_{\ell}^{s_{1}},\bar{z}_{\ell}^{s_{2}})$.
\end{description}
\begin{proof}
This is merely the re-statement of the main theorem of Haller and
Ponsioen \citep{Haller2016} (deduced from Cabré et al. \citep{CabreLlave2003})
for the case of a two-dimensional SSM corresponding to a simple pair
of complex eigenvalues with negative real parts.
\end{proof}

\subsection{Backbone curves and their computation}

When the spectral subspace $\mathcal{E}$ of (\ref{eq:DISCdiag-1})
is lightly damped ($\left|\mathrm{Re}\lambda_{\ell}\right|\ll1$),
the low-order near-resonance relationships 
\[
2\lambda_{\ell}+\bar{\lambda}_{\ell}\approx\lambda_{\ell},\quad\lambda_{\ell}+2\bar{\lambda}_{\ell}\approx\bar{\lambda}_{\ell}
\]
always hold. As in the case of mappings, this prompts us to seek the
polynomial dynamics on the SSM (cf. statement (iv) of Theorem \ref{thm: existence SSM-1})
in the form 
\begin{align}
\dot{\boldsymbol{z}}= & \boldsymbol{R}(\boldsymbol{z})=\begin{pmatrix}\begin{array}{l}
\lambda_{\ell}z_{\ell}+\beta_{\ell}z_{\ell}^{2}\bar{z_{\ell}}\\
\bar{\lambda}_{\ell}\bar{z_{\ell}}+\bar{\beta_{\ell}}z_{\ell}\bar{z_{\ell}}^{2}
\end{array}\end{pmatrix}.\label{eq:1DOFred-1}
\end{align}

Introducing polar coordinates $z=re^{i\theta}$, we can further transform
(\ref{eq:1DOFred-1}) to the real amplitude-phase equations 
\begin{eqnarray}
\dot{\rho} & = & \rho\left(\mathrm{Re}\lambda_{\ell}+\mathrm{Re}\beta_{\ell}\rho^{2}\right),\label{eq:RadAmpl-1}\\
\dot{\theta} & = & \mathrm{Im}\lambda_{\ell}+\mathrm{Im}\beta_{\ell}\rho^{2}.\label{eq:RadAngle-1}
\end{eqnarray}

Equation (\ref{eq:RadAngle-1}) gives instantaneous frequency of nonlinear
oscillations as 
\begin{equation}
\omega(\rho)=\mathrm{Im}\lambda_{\ell}+\mathrm{Im}\beta_{\ell}\rho^{2},\label{eq:omega cont}
\end{equation}
whereas as instantaneous amplitude $\mathrm{Amp}(\rho)$ of the vibration
can be calculated as 
\begin{equation}
\mathrm{Amp}(\rho)=\frac{1}{2\pi}\intop_{0}^{2\pi}\left|\boldsymbol{V}\boldsymbol{\,W}(\boldsymbol{z}(\rho,\theta))\right|\,d\theta,\label{eq:closed curve-1}
\end{equation}
where $\boldsymbol{W}$ is the mapping featured in statement (iii)
of Theorem \ref{thm: existence SSM-1}, and $\boldsymbol{V}$ is the
linear mapping that transform the original, first-order dynamical
system to its standard complex form (\ref{eq:DISCdiag-1}). With the
quantities defined in (\ref{eq:omega cont}) and (\ref{eq:closed curve-1}),
the definition of a backbone curve $\mathcal{B}_{\ell}$ given in
Definition (\ref{def:backbone}) carries over without change to our
present context. Again, to compute the backbone curve (\ref{eq:BackGraph}),
we need to find expressions for the complex coefficient $\beta_{\ell}$
and the mapping $\boldsymbol{W}(\boldsymbol{z})$, as the eigenvalue
$\lambda_{\ell}$ is assumed to be known. 

To this end, we seek the Taylor series coefficients of the $j^{th}$
coordinate functions, $W_{j}(\boldsymbol{z})\in\mathbb{C}$, $j=1,\ldots,2\nu,$
of the mapping $\boldsymbol{W}(\boldsymbol{z})$, and the third-order
Taylor coefficient $\beta_{\ell}\in\mathbb{C}$ of the polynomial
function $\boldsymbol{R}(\boldsymbol{z})$ defined in (\ref{eq:1DOFred-1}).
These unknowns will again be expressed as functions of the $j^{th}$
coordinate functions $G_{j}(\boldsymbol{y})\in\mathbb{C},$ $j=1,\ldots,2\nu,$
of the nonlinear part $\boldsymbol{G}(\boldsymbol{y})$ of the right-hand
side of the dynamical system (\ref{eq:DISCdiag-1}). Using the same
notation as in Theorem (\ref{thm:coefficients-1}), we obtain the
following expressions for the required Taylor coefficients. 
\begin{thm}
\label{thm:coefficients-1} Suppose that the assumptions of Theorem
\ref{thm: existence SSM-1} hold but with the strengthened version
\begin{equation}
s_{1}\lambda_{\ell}+s_{2}\bar{\lambda}_{\ell}\not\approx\lambda_{j},\qquad\forall j\neq\ell,\ell+1,\qquad1\le s_{1}+s_{2}\le\sigma\left(\mathcal{E}\right)\label{eq:nonresonance-2-1}
\end{equation}
of the external non-resonance condition (\ref{eq:nonresonance-1}).
Then, for any $j\in[1,2\nu]$, the $j^{th}$ coordinate function $W_{j}$
of the mapping $\boldsymbol{W}$ in (\ref{eq:W embedding-1}) and
the cubic Taylor coefficient $\beta_{\ell}$ of the conjugate map
$\boldsymbol{R}$ in (\ref{eq:invariancecond0-1}) are given by the
following formulas:

\[
w_{j}^{(1,0)}=\delta_{j\ell},\qquad w_{j}^{(0,1)}=\delta_{j(\ell+1)},
\]
\[
w_{j}^{(2,0)}=\frac{g_{j}^{\left(2@\ell\right)}}{2\lambda_{l}-\lambda_{j}},\qquad w_{j}^{(1,1)}=\frac{g_{j}^{\left(1@\ell,1@\left(\ell+1\right)\right)}}{\lambda_{\ell}+\bar{\lambda}_{\ell}-\lambda_{j}},\qquad w_{j}^{(0,2)}=\frac{g_{j}^{\left(2@\left(\ell+1\right)\right)}}{2\bar{\lambda}_{l}-\lambda_{j}},
\]
\[
w_{j}^{(3,0)}=\frac{\sum_{q=1}^{2\nu}\left(1+\delta_{\ell q}\right)g_{j}^{(1@\ell,1@q)}w_{q}^{(2,0)}+g_{j}^{(3@\ell)}}{3\lambda_{l}-\lambda_{j}},\qquad w_{j}^{(0,3)}=\frac{\sum_{q=1}^{2\nu}\left(1+\delta_{(\ell+1)q}\right)g_{j}^{(1@(\ell+1),1@q)}w_{q}^{(0,2)}+g_{j}^{(3@\left(\ell+1\right))}}{3\bar{\lambda}_{l}-\lambda_{j}}.
\]
 
\[
w_{j}^{(2,1)}=\left(1-\delta_{j\ell}\right)\frac{\sum_{q=1}^{2\nu}\left[\left(1+\delta_{\ell q}\right)g_{j}^{(1@\ell,1@q)}w_{q}^{(1,1)}+\left(1+\delta_{(\ell+1)q}\right)g_{j}^{(1@(\ell+1),1@q)}w_{q}^{(2,0)}\right]+g_{j}^{(2@\ell,1@(\ell+1))}}{2\lambda_{\ell}+\bar{\lambda}_{\ell}-\lambda_{j}},
\]
\[
w_{j}^{(1,2)}=\left(1-\delta_{j(\ell+1)}\right)\frac{\sum_{q=1}^{2\nu}\left[\left(1+\delta_{\ell q}\right)g_{j}^{(1@\ell,1@q)}w_{q}^{(0,2)}+\left(1+\delta_{(\ell+1)q}\right)g_{j}^{(1@(\ell+1),1@q)}w_{q}^{(1,1)}\right]+g_{j}^{(2@(\ell+1),1@\ell)}}{\lambda_{\ell}+2\bar{\lambda}_{\ell}-\lambda_{j}},
\]
 
\[
\beta_{\ell}=\sum_{q=1}^{2\nu}\left[\left(1+\delta_{\ell q}\right)g_{\ell}^{(1@\ell,1@q)}w_{q}^{(1,1)}+\left(1+\delta_{(\ell+1)q}\right)g_{\ell}^{(1@(\ell+1),1@q)}w_{q}^{(2,0)}\right]+g_{\ell}^{(2@\ell,1@(\ell+1))}.
\]
\end{thm}
\begin{proof}
The algebraic equation (\ref{eq:invariancecond0-1}) is similar to
the equation (\ref{eq:invariancecond0}), which we have solved in
detail up to cubic order in the proof of Theorem (\ref{thm:coefficients}).
The first difference between the two equations is that the term $\boldsymbol{\Lambda W}$
in (\ref{eq:invariancecond0}) has the $j^{th}$ component
\begin{equation}
\left(\boldsymbol{\Lambda W}\right)_{j}=\lambda_{j}\sum_{\left|\boldsymbol{s}\right|\geq1}w_{j}^{\boldsymbol{s}}\boldsymbol{z}^{\boldsymbol{s}},\;\boldsymbol{s}\in\mathbb{N}^{2},\quad\quad w_{j}^{\boldsymbol{s}}\in\mathbb{C}.\label{eq:change1}
\end{equation}
The second difference is that instead of $\left(\boldsymbol{W}\circ\boldsymbol{R}\right)_{j}$,
the $j^{th}$ coordinate component of the right-hand side of (\ref{eq:invariancecond0})
is given by
\begin{eqnarray}
\left(D\boldsymbol{W}\circ\boldsymbol{R}\right)_{j} & = & \partial_{z_{\ell}}W_{j}(\boldsymbol{z})\left(\lambda_{\ell}z_{\ell}+\beta_{\ell}z_{\ell}^{2}\bar{z_{\ell}}\right)+\partial_{\bar{z}_{\ell}}W_{j}(\boldsymbol{z})\left(\bar{\lambda}_{\ell}\bar{z_{\ell}}+\bar{\beta_{\ell}}z_{\ell}\bar{z_{\ell}}^{2}\right)\nonumber \\
 & = & \lambda_{\ell}w_{j}^{(1,0)}z_{\ell}+2\lambda_{\ell}w_{j}^{(2,0)}z_{\ell}^{2}+\lambda_{\ell}w_{j}^{(1,1)}z_{\ell}\bar{z}_{\ell}+3\lambda_{\ell}w_{j}^{(3,0)}z_{\ell}^{3}\nonumber \\
 &  & +2\lambda_{\ell}w_{j}^{(2,1)}z_{\ell}^{2}\bar{z}_{\ell}+\lambda_{\ell}w_{j}^{(1,2)}z_{\ell}\bar{z}_{\ell}^{2}+\beta_{\ell}w_{j}^{(1,0)}z_{\ell}^{2}\bar{z}_{\ell}+\mathcal{O}\left(\left|z\right|^{4}\right)\nonumber \\
 &  & +\bar{\lambda}_{\ell}w_{j}^{(0,1)}\bar{z}_{\ell}+2\bar{\lambda}_{\ell}w_{j}^{(0,2)}\bar{z}_{\ell}^{2}+\bar{\lambda}_{\ell}w_{j}^{(1,1)}z_{\ell}\bar{z}_{\ell}+3\bar{\lambda}_{\ell}w_{j}^{(0,3)}\bar{z}_{\ell}^{3}\nonumber \\
 &  & +2\bar{\lambda}_{\ell}w_{j}^{(1,2)}z_{\ell}\bar{z}_{\ell}^{2}+\bar{\lambda}_{\ell}w_{j}^{(2,1)}z_{\ell}^{2}\bar{z}_{\ell}+\bar{\beta}_{\ell}w_{j}^{(0,1)}z_{\ell}\bar{z}_{\ell}^{2}+\mathcal{O}\left(\left|z\right|^{4}\right).\label{eq:change2}
\end{eqnarray}
Substituting formulas (\ref{eq:change1})-(\ref{eq:change2}) into
(\ref{eq:invariancecond0}), and using the expression for $\left(\boldsymbol{W}\circ\boldsymbol{G}\right)_{j}$
from the proof of Theorem (\ref{thm:coefficients}), we obtain the
formulas in the statement of Theorem (\ref{thm:coefficients-1}) after
comparing equal powers of $\boldsymbol{z}$ up to cubic order.
\end{proof}

\section{Appendix C: Analytic SSM and backbone calculations for Example 1}

To compute the SSMs $W(E_{1})$and $W(E_{3})$ in Example 1, we transform
(\ref{eq:ShawPierreModel}) to its complex standard form (\ref{eq:DISCdiag-1}).
This involves the coordinate change $\boldsymbol{x}=\left(x_{1},x_{2},v_{1},v_{2}\right)^{T}=\boldsymbol{V}\boldsymbol{y}$,
where the matrix $\boldsymbol{V}$ of eigenvectors and its inverse
are 
\[
\boldsymbol{V}=\left(\begin{array}{ccrr}
1 & 1 & 1 & 1\\
1 & 1 & -1 & -1\\
\lambda_{1} & \overline{\lambda}_{1} & \lambda_{3} & \overline{\lambda}_{3}\\
\lambda_{1} & \overline{\lambda}_{1} & -\lambda_{3} & -\overline{\lambda}_{3}
\end{array}\right),\qquad\boldsymbol{V}^{-1}=\left(\begin{array}{cccr}
-\frac{\overline{\lambda}_{1}}{2\,\left(\lambda_{1}-\overline{\lambda}_{1}\right)} & -\frac{\overline{\lambda}_{1}}{2\,\left(\lambda_{1}-\overline{\lambda}_{1}\right)} & \frac{1}{2\,\left(\lambda_{1}-\overline{\lambda}_{1}\right)} & \frac{1}{2\,\left(\lambda_{1}-\overline{\lambda}_{1}\right)}\\
\frac{\lambda_{1}}{2\,\left(\lambda_{1}-\overline{\lambda}_{1}\right)} & \frac{\lambda_{1}}{2\,\left(\lambda_{1}-\overline{\lambda}_{1}\right)} & -\frac{1}{2\,\left(\lambda_{1}-\overline{\lambda}_{1}\right)} & -\frac{1}{2\,\left(\lambda_{1}-\overline{\lambda}_{1}\right)}\\
-\frac{\overline{\lambda}_{3}}{2\,\left(\lambda_{3}-\overline{\lambda}_{3}\right)} & \frac{\overline{\lambda}_{3}}{2\,\left(\lambda_{3}-\overline{\lambda}_{3}\right)} & \frac{1}{2\,\left(\lambda_{3}-\overline{\lambda}_{3}\right)} & -\frac{1}{2\,\left(\lambda_{3}-\overline{\lambda}_{3}\right)}\\
\frac{\lambda_{3}}{2\,\left(\lambda_{3}-\overline{\lambda}_{3}\right)} & -\frac{\lambda_{3}}{2\,\left(\lambda_{3}-\overline{\lambda}_{3}\right)} & -\frac{1}{2\,\left(\lambda_{3}-\overline{\lambda}_{3}\right)} & \frac{1}{2\,\left(\lambda_{3}-\overline{\lambda}_{3}\right)}
\end{array}\right).
\]
The transformed system (\ref{eq:ShawPierreModel}) then takes the
form
\begin{eqnarray*}
\dot{\boldsymbol{y}} & = & \boldsymbol{\Lambda}\boldsymbol{y}+\boldsymbol{G}(\boldsymbol{y}),\qquad\boldsymbol{y}\in\mathbb{C}^{4},\qquad\boldsymbol{\Lambda}=\mathrm{diag}(\lambda_{1},\lambda_{2},\lambda_{3},\lambda_{4}),\quad\boldsymbol{G}(\boldsymbol{y})=\mathcal{O}\left(\left|\boldsymbol{y}\right|^{2}\right),\\
\boldsymbol{G}(\boldsymbol{y}) & = & \boldsymbol{V}^{-1}\boldsymbol{f}(\boldsymbol{V}\boldsymbol{y})=-\frac{\kappa}{2}\gamma(\boldsymbol{y})\left(\begin{array}{c}
\frac{1}{\lambda_{1}-\overline{\lambda}_{1}}\\
\frac{1}{\overline{\lambda}_{1}-\lambda_{1}}\\
\frac{1}{\lambda_{3}-\overline{\lambda}_{3}}\\
\frac{1}{\overline{\lambda}_{3}-\lambda_{3}}
\end{array}\right)=\frac{i\kappa\,\gamma(\boldsymbol{y})}{4}\left(\begin{array}{c}
\frac{1}{\mathrm{Im\,}\lambda_{1}}\\
-\frac{1}{\mathrm{Im\,}\lambda_{1}}\\
\frac{1}{\mathrm{Im\,}\lambda_{3}}\\
-\frac{1}{\mathrm{Im\,}\lambda_{3}}
\end{array}\right),
\end{eqnarray*}
where
\begin{eqnarray*}
\gamma(\boldsymbol{y}) & = & y_{1}^{3}+y_{2}^{3}+y_{3}^{3}+y_{4}^{3}\\
 &  & +3(y_{1}^{2}y_{2}+y_{1}^{2}y_{3}+y_{1}^{2}y_{4}+y_{2}^{2}y_{1}+y_{2}^{2}y_{3}+y_{2}^{2}y_{4}+y_{3}^{2}y_{1}+y_{3}^{2}y_{2}+y_{3}^{2}y_{4}+y_{4}^{2}y_{1}+y_{4}^{2}y_{2}+y_{4}^{2}y_{3})\\
 &  & +6(y_{1}y_{2}y_{3}+y_{1}y_{2}y_{4}+y_{1}y_{3}y_{4}+y_{2}y_{3}y_{4})
\end{eqnarray*}
Theorem 4 of Appendix B then gives the following coefficients for
$\ell=1$: 
\[
w_{1}^{(1,0)}=1,\quad w_{2}^{(1,0)}=w_{3}^{(1,0)}=w_{4}^{(1,0)}=0,\quad w_{2}^{(0,1)}=1,\quad w_{1}^{(0,1)}=w_{3}^{(0,1)}=w_{4}^{(0,1)}=0,
\]
\\
\[
w_{j}^{(2,0)}=w_{j}^{(1,1)}=w_{j}^{(0,2)}=0,\qquad j=1,2,3,4,
\]
\\
\begin{eqnarray*}
w_{1}^{(3,0)} & = & \frac{i\kappa}{8\lambda_{1}\mathrm{Im\,}\lambda_{1}},\quad\qquad\qquad w_{2}^{(3,0)}=-\frac{i\kappa}{4\left(3\lambda_{1}-\bar{\lambda}_{1}\right)\mathrm{Im\,}\lambda_{1}},\\
w_{3}^{(3,0)} & = & \frac{i\kappa}{4\left(3\lambda_{1}-\lambda_{3}\right)\mathrm{Im\,}\lambda_{3}},\quad w_{4}^{(3,0)}=-\frac{i\kappa}{4\left(3\lambda_{1}-\bar{\lambda}_{3}\right)\mathrm{Im\,}\lambda_{3}},\\
w_{1}^{(0,3)} & = & \frac{i\kappa}{4\left(3\bar{\lambda}_{1}-\lambda_{1}\right)\mathrm{Im\,}\lambda_{1}},\quad w_{2}^{(0,3)}=-\frac{i\kappa}{8\bar{\lambda}_{1}\mathrm{Im\,}\lambda_{1}},\quad\\
w_{3}^{(0,3)} & = & \frac{i\kappa}{4\left(3\bar{\lambda}_{1}-\lambda_{3}\right)\mathrm{Im\,}\lambda_{3}},\quad w_{4}^{(0,3)}=-\frac{i\kappa}{4\left(3\bar{\lambda}_{1}-\bar{\lambda}_{3}\right)\mathrm{Im\,}\lambda_{3}},
\end{eqnarray*}
 
\[
w_{1}^{(2,1)}=0,\quad w_{2}^{(2,1)}=\frac{-i3\kappa}{8\lambda_{1}\mathrm{Im\,}\lambda_{1}},\quad w_{3}^{(2,1)}=\frac{i3\kappa}{4\left(2\lambda_{1}+\bar{\lambda}_{1}-\lambda_{3}\right)\mathrm{Im\,}\lambda_{3}},\quad w_{4}^{(2,1)}=\frac{-i3\kappa}{4\left(2\lambda_{1}+\bar{\lambda}_{1}-\bar{\lambda}_{3}\right)\mathrm{Im\,}\lambda_{3}},
\]
\[
w_{1}^{(1,2)}=\frac{i3\kappa}{8\bar{\lambda}_{1}\mathrm{Im\,}\lambda_{1}},\quad w_{2}^{(1,2)}=0,\quad w_{3}^{(1,2)}=\frac{i3\kappa}{4\left(\lambda_{1}+2\bar{\lambda}_{1}-\lambda_{3}\right)\mathrm{Im\,}\lambda_{3}},\quad w_{4}^{(1,2)}=\frac{-i3\kappa}{4\left(\lambda_{1}+2\bar{\lambda}_{1}-\bar{\lambda}_{3}\right)\mathrm{Im\,}\lambda_{3}},
\]
 
\[
\beta_{1}=\frac{i3\kappa}{4\mathrm{Im\,}\lambda_{1}}.
\]
The coefficients for $\ell=2$ are the complex conjugates of the above.
The transformation for the SSM $W(\mathcal{E}_{1})$ up to cubic order
is therefore of the form
\begin{equation}
\boldsymbol{W}(z_{1},\bar{z}_{1})=\begin{pmatrix}\begin{array}{l}
z_{1}+\frac{i\kappa z_{1}^{3}}{8\lambda_{1}\mathrm{Im\,}\lambda_{1}}+\frac{i\kappa\bar{z}_{1}^{3}}{4\left(3\bar{\lambda}_{1}-\lambda_{1}\right)\mathrm{Im\,}\lambda_{1}}+\frac{i3\kappa z_{1}\bar{z}_{1}^{2}}{8\bar{\lambda}_{1}\mathrm{Im\,}\lambda_{1}}\\
\bar{z}_{1}-\frac{i\kappa z_{1}^{3}}{4\left(3\lambda_{1}-\bar{\lambda}_{1}\right)\mathrm{Im\,}\lambda_{1}}-\frac{i\kappa\bar{z}_{1}^{3}}{8\bar{\lambda}_{1}\mathrm{Im\,}\lambda_{1}}-\frac{i3\kappa z_{1}^{2}\bar{z}_{1}}{8\lambda_{1}\mathrm{Im\,}\lambda_{1}}\\
\frac{i\kappa z_{1}^{3}}{4\left(3\lambda_{1}-\lambda_{3}\right)\mathrm{Im\,}\lambda_{3}}+\frac{i\kappa\bar{z}_{1}^{3}}{4\left(3\bar{\lambda}_{1}-\lambda_{3}\right)\mathrm{Im\,}\lambda_{3}}+\frac{i3\kappa z_{1}^{2}\bar{z}_{1}}{4\left(2\lambda_{1}+\bar{\lambda}_{1}-\lambda_{3}\right)\mathrm{Im\,}\lambda_{1}}+\frac{i3\kappa z_{1}\bar{z}_{1}^{2}}{4\left(\lambda_{1}+2\bar{\lambda}_{1}-\lambda_{3}\right)\mathrm{Im\,}\lambda_{3}}\\
-\frac{i\kappa z_{1}^{3}}{4\left(3\lambda_{1}-\bar{\lambda}_{3}\right)\mathrm{Im\,}\lambda_{3}}-\frac{i\kappa\bar{z}_{1}^{3}}{4\left(3\bar{\lambda}_{1}-\bar{\lambda}_{3}\right)\mathrm{Im\,}\lambda_{3}}-\frac{i3\kappa z_{1}^{2}\bar{z}_{1}}{4\left(2\lambda_{1}+\bar{\lambda}_{1}-\bar{\lambda}_{3}\right)\mathrm{Im\,}\lambda_{3}}-\frac{i3\kappa z_{1}\bar{z}_{1}^{2}}{4\left(\lambda_{1}+2\bar{\lambda}_{1}-\bar{\lambda}_{3}\right)\mathrm{Im\,}\lambda_{3}}
\end{array}\end{pmatrix}.\label{eq:ShPiW-1-1}
\end{equation}
Passing to polar coordinates via the substitution $z_{1}=\rho_{1}e^{i\theta_{1}}$,
the corresponding leading-order dynamics (\ref{eq:RadAmpl-1})-(\ref{eq:RadAngle-1})
on $W(\mathcal{E}_{1})$ is given by the equations 
\begin{eqnarray}
\dot{\rho_{1}} & = & -\frac{c}{2}\rho_{1},\label{eq:RadAmpl-1-1-1}\\
\dot{\theta_{1}} & = & \frac{1}{2}\sqrt{4k_{0}-c^{2}}+\frac{3\kappa}{2\sqrt{4k_{0}-c^{2}}}\rho_{1}^{2}.\label{eq:RadAngle-1-1-1}
\end{eqnarray}
As obtained in formula (\ref{eq:omega cont}), the instantaneous frequency
of the oscillations on the SSM $\mathcal{E}_{1}$ is then
\[
\omega(\rho_{1})=\frac{1}{2}\left(\sqrt{4k_{0}-c^{2}}+\frac{3\kappa}{\sqrt{4k_{0}-c^{2}}}\rho_{1}^{2}\right).
\]
 The squared $L^{2}$ norm of the amplitude in the original $x$ coordinates
\begin{eqnarray*}
\left[\mathrm{Amp}(\rho_{1})\right]^{2} & = & \frac{1}{2\pi}\int_{0}^{2\pi}\left|\boldsymbol{V_{x}W}(\rho_{1}e^{i\theta_{1}},\rho_{1}e^{-i\theta_{1}})\right|^{2}d\theta=\\
 & = & \frac{1}{2\pi}\int_{0}^{2\pi}\left|W_{1}+\bar{W}_{1}+W_{3}+\bar{W}_{3}\right|^{2}d\theta+\frac{1}{2\pi}\int_{0}^{2\pi}\left|W_{1}+\bar{W}_{1}-W_{3}-\bar{W}_{3}\right|^{2}d\theta\\
 & = & \frac{1}{2\pi}\int_{0}^{2\pi}\left|z_{1}+\bar{z}_{1}+\mathcal{O}\left(\left|\boldsymbol{z}\right|^{3}\right)\right|^{2}d\theta+\frac{1}{2\pi}\int_{0}^{2\pi}\left|z_{1}+\bar{z}_{1}+\mathcal{O}\left(\left|\boldsymbol{z}\right|^{3}\right)\right|^{2}d\theta\\
 & = & \frac{1}{2\pi}\int_{0}^{2\pi}\left[4z_{1}\bar{z}_{1}+\mathcal{O}\left(\left|\boldsymbol{z}\right|^{4}\right)\right]d\theta=4\rho_{1}^{2}+\mathcal{O}\left(\rho_{1}^{4}\right),
\end{eqnarray*}
which gives $\mathrm{Amp}(\rho_{1})\approx2\rho_{1}$ for small $\rho_{1}$.
For higher values of $\rho_{1}$, the exact dependence form of $\mathrm{Amp}(\rho_{1})$
can be found by evaluating the above integral numerically.

Similarly, for $\ell=3,$ Theorem 1 gives the coefficients 
\[
w_{3}^{(1,0)}=1,\quad w_{1}^{(1,0)}=w_{2}^{(1,0)}=w_{4}^{(1,0)}=0,\quad w_{4}^{(0,1)}=1,\quad w_{1}^{(0,1)}=w_{2}^{(0,1)}=w_{3}^{(0,1)}=0,
\]
\[
w_{j}^{(2,0)}=w_{j}^{(1,1)}=w_{j}^{(0,2)}=0,\qquad j=1,2,3,4,
\]

\begin{eqnarray*}
w_{1}^{(3,0)} & = & \frac{i\kappa}{4\left(3\lambda_{3}-\lambda_{1}\right)\mathrm{Im\,}\lambda_{1}},\quad w_{2}^{(3,0)}=-\frac{i\kappa}{4\left(3\lambda_{3}-\bar{\lambda}_{1}\right)\mathrm{Im\,}\lambda_{1}},\quad\\
 &  & w_{3}^{(3,0)}=\frac{i\kappa}{8\lambda_{3}\mathrm{Im\,}\lambda_{3}},\quad w_{4}^{(3,0)}=-\frac{i\kappa}{4\left(3\lambda_{3}-\bar{\lambda}_{3}\right)\mathrm{Im\,}\lambda_{3}},\\
w_{1}^{(0,3)} & = & \frac{i\kappa}{4\left(3\bar{\lambda}_{3}-\lambda_{1}\right)\mathrm{Im\,}\lambda_{1}},\quad w_{2}^{(0,3)}=-\frac{i\kappa}{4\left(3\bar{\lambda}_{3}-\bar{\lambda}_{1}\right)\mathrm{Im\,}\lambda_{1}},\quad\\
 &  & w_{3}^{(0,3)}=\frac{i\kappa}{4\left(3\bar{\lambda}_{3}-\lambda_{3}\right)\mathrm{Im\,}\lambda_{3}},\quad w_{4}^{(0,3)}=-\frac{i\kappa}{8\bar{\lambda}_{3}\mathrm{Im\,}\lambda_{3}},
\end{eqnarray*}
 
\[
w_{1}^{(2,1)}=\frac{i3\kappa}{4\left(2\lambda_{3}+\bar{\lambda}_{3}-\lambda_{1}\right)\mathrm{Im\,}\lambda_{1}},\quad w_{2}^{(2,1)}=\frac{-i3\kappa}{4\left(2\lambda_{3}+\bar{\lambda}_{3}-\bar{\lambda}_{1}\right)\mathrm{Im\,}\lambda_{1}},\quad w_{3}^{(2,1)}=0,\quad w_{4}^{(2,1)}=\frac{-i3\kappa}{8\lambda_{3}\mathrm{Im\,}\lambda_{3}},
\]
\[
w_{1}^{(1,2)}=\frac{i3\kappa}{4\left(\lambda_{3}+2\bar{\lambda}_{3}-\lambda_{1}\right)\mathrm{Im\,}\lambda_{1}},\quad w_{2}^{(1,2)}=\frac{-i3\kappa}{4\left(\lambda_{3}+2\bar{\lambda}_{3}-\bar{\lambda}_{1}\right)\mathrm{Im\,}\lambda_{1}},\quad w_{3}^{(1,2)}=\frac{i3\kappa}{8\bar{\lambda}_{3}\mathrm{Im\,}\lambda_{3}},\quad w_{4}^{(1,2)}=0,
\]
 
\[
\beta_{3}=\frac{i3\kappa}{4\mathrm{Im\,}\lambda_{3}},
\]
whose complex conjugates are the corresponding coefficients for $\ell=4.$ 

Similarly, the transformation for the SSM $W(\mathcal{E}_{2})$ up
to cubic order is of the form
\begin{equation}
\boldsymbol{W}(z_{3},\bar{z}_{3})=\begin{pmatrix}\begin{array}{l}
\frac{i\kappa z_{3}^{3}}{4\left(3\lambda_{3}-\lambda_{1}\right)\mathrm{Im\,}\lambda_{1}}+\frac{i\kappa\bar{z}_{3}^{3}}{4\left(3\bar{\lambda}_{3}-\lambda_{1}\right)\mathrm{Im\,}\lambda_{1}}+\frac{i3\kappa z_{3}^{2}\bar{z}_{3}}{4\left(2\lambda_{3}+\bar{\lambda}_{3}-\lambda_{1}\right)\mathrm{Im\,}\lambda_{1}}+\frac{i3\kappa z_{3}\bar{z}_{3}^{2}}{4\left(\lambda_{3}+2\bar{\lambda}_{3}-\lambda_{1}\right)\mathrm{Im\,}\lambda_{1}}\\
-\frac{i\kappa z_{3}^{3}}{4\left(3\lambda_{3}-\bar{\lambda}_{1}\right)\mathrm{Im\,}\lambda_{1}}-\frac{i\kappa\bar{z}_{3}^{3}}{4\left(3\bar{\lambda}_{3}-\bar{\lambda}_{1}\right)\mathrm{Im\,}\lambda_{1}}-\frac{i3\kappa z_{3}^{2}\bar{z}_{3}}{4\left(2\lambda_{3}+\bar{\lambda}_{3}-\bar{\lambda}_{1}\right)\mathrm{Im\,}\lambda_{1}}-\frac{i3\kappa z_{3}\bar{z}_{3}^{2}}{4\left(\lambda_{3}+2\bar{\lambda}_{3}-\bar{\lambda}_{1}\right)\mathrm{Im\,}\lambda_{1}}\\
z_{3}+\frac{i\kappa z_{3}^{3}}{8\lambda_{3}\mathrm{Im\,}\lambda_{3}}+\frac{i\kappa\bar{z}_{3}^{3}}{4\left(3\bar{\lambda}_{3}-\lambda_{3}\right)\mathrm{Im\,}\lambda_{3}}+\frac{i3\kappa z_{3}\bar{z}_{3}^{2}}{8\bar{\lambda}_{3}\mathrm{Im\,}\lambda_{3}}\\
\bar{z}_{3}-\frac{i\kappa z_{3}^{3}}{4\left(3\lambda_{3}-\bar{\lambda}_{3}\right)\mathrm{Im\,}\lambda_{3}}-\frac{i\kappa\bar{z}_{3}^{3}}{8\bar{\lambda}_{3}\mathrm{Im\,}\lambda_{3}}-\frac{i3\kappa z_{3}^{2}\bar{z}_{3}}{8\lambda_{3}\mathrm{Im\,}\lambda_{3}}
\end{array}\end{pmatrix},\label{eq:ShPiW-1}
\end{equation}
with the corresponding dynamics on the first SSM is described by (cf.
(\ref{eq:RadAmpl-1})-(\ref{eq:RadAngle-1})) 
\begin{eqnarray}
\dot{\rho}_{3} & = & -\frac{3c}{2}\rho_{3},\label{eq:RadAmpl-1-1}\\
\dot{\theta}_{3} & = & \frac{1}{2}\sqrt{3\left(4k_{0}-3c^{2}\right)}+\frac{\sqrt{3}\kappa}{2\sqrt{4k_{0}-3c^{2}}}\rho_{3}^{2}.\label{eq:RadAngle-1-1}
\end{eqnarray}
Following the same calculation as for $\ell=1$, we obtain 
\begin{eqnarray*}
\omega(\rho_{3}) & = & \frac{1}{2}\left(\sqrt{3\left(4k_{0}-3c^{2}\right)}+\frac{\sqrt{3}\kappa}{\sqrt{4k_{0}-3c^{2}}}\rho_{3}^{2}\right),\\
\mathrm{Amp}(\rho_{3}) & \approx & 2\rho_{3}.
\end{eqnarray*}

\bibliographystyle{plainurl}
\bibliography{AllRef}

\end{document}